\documentclass[a4paper,10pt]{article}
\usepackage{authblk}

\usepackage[ngerman,english]{babel}
\usepackage[latin1]{inputenc}
\usepackage{csquotes}

\usepackage{amsfonts,amsmath,amsthm}
\usepackage{empheq}
\usepackage{enumitem}

\usepackage[titletoc,title]{appendix}

\usepackage[backend=bibtex8,url=true,doi=false,eprint=false,giveninits=true,isbn=false,style=numeric-comp,maxnames=99]{biblatex}
\makeatletter
\def\blx@maxline{77}
\makeatother

\bibliography{bibl.bib}
\AtEveryBibitem{\clearlist{language}}

\usepackage{cases}
\usepackage{mathabx}
\usepackage{bbm}
\usepackage{xfrac}
\usepackage{slashbox}

\usepackage{caption}
\usepackage
{subcaption}

\usepackage{xstring}

\usepackage{fancyhdr}
\usepackage{pgffor}

\usepackage{color}
\usepackage[colorinlistoftodos,textsize=small,backgroundcolor=white,bordercolor=magenta,linecolor=magenta]{todonotes}
\usepackage[colorlinks]{hyperref}
\definecolor{blue75}{rgb}{0,0,.75}
\definecolor{green75}{rgb}{0,.75,0}
\hypersetup{colorlinks=true, urlcolor=blue75,linkcolor=blue75,citecolor=green75,pdfstartview=FitB,bookmarksopen=true,bookmarksopenlevel=1}

\usepackage[a4paper, left=2.5cm, right=2.5cm, top=2.5cm,bottom=2cm]{geometry}
\usepackage{constants}
\newcommand{\parenthezises}[1]{\arabic{#1}}
\newconstantfamily{C}{
symbol=C,
format=\parenthezises,
}
\newconstantfamily{F}{
symbol=F,
format=\parenthezises,
}
\newconstantfamily{G}{
symbol=G,
format=\parenthezises,
reset={section}
}

\newconstantfamily{M}{
symbol=M,
format=\parenthezises,
reset={section}
}
\newconstantfamily{B}{
symbol=B,
format=\parenthezises,
reset={section}
}
\newconstantfamily{xi}{
symbol=\xi,
format=\parenthezises,
reset={section}
}
\usepackage{graphicx}
\usepackage{wrapfig}
\usepackage{figbib}
\allowdisplaybreaks

\usepackage[capitalise]{cleveref}

\crefdefaultlabelformat{{\it #2#1#3}}

\crefname{equation}{}{}

\crefname{enumi}{}{}
\crefformat{enumi}{#2#1#3}

\crefname{section}{{\it Section}}{{\it Sections}}
\crefname{subsection}{{\it Section}}{{\it Sections}}
\crefname{subsubsection}{{\it Paragraph}}{{\it Paragraphs}}
\crefname{table}{{\it Table}}{{\it Tables}}

\newtheorem{Theorem}{Theorem}[section]
\crefname{Theorem}{{\it Theorem}}{{\it Theorems}}
\newtheorem{Definition}[Theorem]{Definition}
\crefname{Definition}{{\it Definition}}{{\it Definitions}}
\newtheorem{Lemma}[Theorem]{Lemma}
\crefname{Lemma}{{\it Lemma}}{{\it Lemmas}}

\newtheorem{Corollary}[Theorem]{Corollary}
\crefname{Proposition}{{\it Proposition}}{{\it Propositions}}

\crefname{Assumptions}{{\it Assumptions}}{{\it Assumptions}}

\theoremstyle{definition}
\newtheorem{Remark}[Theorem]{Remark}
\crefname{Remark}{{\it Remark}}{{\it Remarks}}

\crefname{Notation}{{\it Notation}}{{\it Notations}}

\crefname{Example}{{\it Example}}{{\it Examples}}

\usepackage[version=4]{mhchem}

\begin{document}

\newcommand{\D}{\mathbb{D}}
\newcommand{\R}{\mathbb{R}}
\newcommand{\N}{\mathbb{N}}
\newcommand{\Q}{\mathbb{Q}}

\newcommand{\limn}{\underset{n\rightarrow\infty}{\rightarrow}}
\newcommand{\uPi}{\underline{\Pi}}
\newcommand{\oPi}{\overline{\Pi}}

\def\diam{\operatorname{diam}}
\def\dist{\operatorname{dist}}
\def\diver{\operatorname{div}}
\def\ess{\operatorname{ess}}
\def\id{\operatorname{id}}
\def\inner{\operatorname{int}}
\def\osc{\operatorname{osc}}
\def\sign{\operatorname{sign}}
\def\supp{\operatorname{supp}}
\def\exp{\operatorname{exp}}
\newcommand{\BMO}{BMO(\Omega)}
\newcommand{\LOne}{L^{1}(\Omega)}
\newcommand{\LOnen}{(L^{1}(\Omega))^d}
\newcommand{\LTwo}{L^{2}(\Omega)}
\newcommand{\Lq}{L^{q}(\Omega)}
\newcommand{\Lp}{L^{2}(\Omega)}
\newcommand{\Lpn}{(L^{2}(\Omega))^d}
\newcommand{\LInf}{L^{\infty}(\Omega)}
\newcommand{\HOneO}{H^{1,0}(\Omega)}
\newcommand{\HTwoO}{H^{2,0}(\Omega)}
\newcommand{\HOne}{H^{1}(\Omega)}
\newcommand{\HTwo}{H^{2}(\Omega)}
\newcommand{\HmOne}{H^{-1}(\Omega)}
\newcommand{\HmTwo}{H^{-2}(\Omega)}
\newcommand{\WOne}{W^{1,1}(\Omega)}

\newcommand{\LlogL}{L\log L(\Omega)}
\newcommand{\LpX}{L^p((0,T);X)}
\newcommand{\LtwoX}{L^2((0,T);X)}

\newcommand{\Co}{C_0(\overline{\Omega})}
\newcommand{\RM}{{\cal M}(\Omega)}
\newcommand{\BV}{BV(\Omega)}

\newcommand{\A}{{\cal A}}
\newcommand{\B}{{\cal B}}
\newcommand{\Ss}{{\cal S}}
\newcommand{\T}{{\cal T}}
\newcommand{\X}{{\cal X}}

\def\avint{\mathop{\,\rlap{-}\!\!\int}\nolimits}

\theoremstyle{definition}
\newtheorem{proofpart}{Step}
\makeatletter
\@addtoreset{proofpart}{Theorem}
\makeatother
\numberwithin{equation}{section}

\title{Global entropy solutions to a degenerate parabolic-parabolic  chemotaxis system for flux-limited dispersal}
\author{Anna~Zhigun}
\renewcommand\Affilfont{\itshape\small}
\affil{ Queen's University Belfast, School of Mathematics and Physics\\
University Road, Belfast BT7 1NN, Northern Ireland, UK\\
  \href{mailto:A.Zhigun@qub.ac.uk}{A.Zhigun@qub.ac.uk}}
\date{}
\maketitle
\begin{abstract}
  Existence of global finite-time bounded entropy solutions to a parabolic-parabolic  system proposed in \cite{BBNS2010} is established in bounded domains under no-flux boundary conditions for nonnegative bounded initial data. This  modification of the classical Keller-Segel model features  degenerate diffusion and chemotaxis that are both subject to flux-saturation. The approach is based on Schauder's fixed point theorem and calculus of functions of bounded variation. 
\end{abstract}

\section{Introduction}
In this paper we study the  initial-boundary value problem (IBVP) for the chemotaxis system 
\begin{subequations}\label{new}
 \begin{align}
  &\partial_t c=\nabla\cdot\left(D_c\frac{c\nabla c}{\sqrt{c^2+\frac{D_c^2}{C^2}|\nabla c|^2}}-\chi c \frac{\nabla v}{\sqrt{1+|\nabla v|^2}}\right)+f_c(c,v),\label{newc}\\
  &\partial_t v=D_v\Delta v+f_v(c,v)\label{newv}
 \end{align}
 \end{subequations} 
 equipped with no-flux boundary conditions. Equations \cref{new} were derived in \cite{BBNS2010} and serve to overcome several issues arising in connection with the classical Keller-Segel model for chemotaxis  \cite{KS70,KS71}
\begin{subequations}\label{KS}
 \begin{align}
  &\partial_t c=\nabla\cdot(D_c\nabla c-\chi c\nabla v),\label{KSc}\\
  &\partial_t v=D_v\Delta v-\gamma v+\alpha c.\label{KSv}
 \end{align}
\end{subequations}
Both systems describe  spatio-temporal evolution of cell density, $c$ and concentration of  chemoattractant, $v$.  Functions  $f_c$ and $f_v$ and constants $C,D_c,D_v,\chi,\alpha,\gamma>0$  are assumed given. 
Each of the $c$-equations \cref{KSc,newc} can be rewritten in the form of a continuity equation
\begin{align}
 &\partial_t c=-\nabla\cdot(c(V_d+V_c)),\nonumber
\end{align}
where $V_d$ and $V_c$ are velocities due to diffusion and chemotaxis, respectively.
Whereas for \cref{KSc} 
\begin{align}
 V_d=-D_C\frac{\nabla c}{c},\qquad V_c=\chi \nabla v\nonumber
\end{align}
need not be finite everywhere, for \cref{newc} we have 
\begin{align}
 &V_d=-D_c\frac{\nabla c}{\sqrt{c^2+\frac{D_c^2}{C^2}|\nabla c|^2}},\qquad V_c=\chi  \frac{\nabla v}{\sqrt{1+|\nabla v|^2}},\label{Vdc}
\end{align}
so that $V_d$ and $V_c$ are a priori bounded by $C$ and $\chi$, respectively. 
Thus, unlike \cref{KSc}, both diffusion and chemotaxis in \cref{newc} are flux-limited (FL).  
{A justified expectation is that for such choices as in \cref{Vdc} neither the overall propagation speed of the cell  population can  exceed the universal constant $C+\chi$ nor can aggregation due to taxis be unlimited. In contrast, for the original model \cref{KS}, the propagation speed is always infinite due to the non-degeneracy of diffusion, whereas taxis can lead even to blow-ups, i.e. the cell density tending to infinity, in finite or infinite time, see, e.g. \cite{Horstmann,BBTW,LankWink2020}. System \cref{new} therefore provides a more realistic description for cell dispersion due to diffusion and chemotaxis.} In this work, we concentrate solely on the analysis of \cref{new} and refer the interested reader to  \cite{BBNS2010} as well as our review \cite{ZhRevFluxL} for {further} details on the biological motivation behind  equations with FL diffusion and/or taxis and their derivation. 

While an equation with a flux-limited diffusion  may be attractive for modelling, the presence of such a term   leads to substantial analytical difficulties. In particular, spatial derivatives of solutions that are nonnegative but not everywhere positive are, in general, Radon measures, and not some integrable functions. Dealing with such  solutions requires using the calculus of functions of bounded variation (BV functions). 

 Equations with FL diffusion, also with  $V_d$ different from that in \cref{Vdc}, however without chemotaxis, i.e. for $\chi=0$, are already well-understood: well-posedness and other qualitative properties  under various boundary conditions and in the whole space were established in \cite{ACM021,ACM02,ACM05,ACM05Cauchy,ACMM06,ACM08,ACMM10,ACalvoMSoler,CMSV}, see also the review paper \cite{CCCSS15} for further references. Well-posedness of the Cauchy problem for a FL reaction-diffusion equation was obtained in  \cite{ACM10} and the corresponding travelling waves were studied in \cite{CGSS13}. 
As to FL diffusion-chemotaxis systems, only parabolic-elliptic modifications, i.e. such where  \cref{newv} is replaced by some elliptic equation, have been analysed so far. 
  Existence and boundedness/blow-up were addressed in the radial-symmetric case for strictly positive initial values of $c$ in \cite{BelWinkler20171,BelWinkler20172,MIZUKAMI2019,Chiyoda2019,TuMuZheng2022,YiMuQiuXu2021}. The travelling waves analysis in   \cite{arias2018cross,CYPerthame,campos2021kinks} allowed for biologically relevant not necessarily strictly positive wave profiles.  A qualitative study of a fully parabolic system such as \cref{new} has been lacking. 
In the present work we make a step in this direction by establishing for $$0<\chi<C$$ global existence for the IBVP associated with \cref{new} subject to no-flux boundary conditions in a certain class of entropy solutions.  Our main  \cref{mainthm} (see below) holds in general spatial domains and for general nonnegative and bounded initial data. No radial symmetry or initial positivity of $c$ is thus required. 

System \cref{new} raises analytical challenges  that differ from those for   \cref{KS}. Indeed, since the latter is a regular  upper-triangular parabolic system, standard  theory (see, e.g. \cite{Amann1}) ensures global well-posedness in the classical sense provided that solutions are a priori bounded for finite times. The study of such systems thus mostly concentrates upon the issue of boundedness/blow-up. In the case of system \cref{new}  sufficiently regular weak solutions are (see \cref{Lemcbnd} below) a priori bounded over finite times if $0<\chi<C$ and  $f_c$ satisfies appropriate assumptions. This holds thanks to the imposed limitation of the chemotactic flux. 
Notwithstanding,   saturation in the diffusion term precludes the application of standard results for parabolic equations and systems. In  \cite{ACM021,ACM02,ACM05,ACM05Cauchy,ACMM06,ACM08,ACMM10}, where a solvability theory  for single parabolic FL diffusion equations was developed, one first constructed a mild solution and then verified that it is actually the unique entropy solution.  In the case of a strongly coupled system such as  \cref{new} it seems that neither it is possible to  set up a semigroup for the system, nor to prove  uniqueness of entropy solutions, mainly due to the generally poor regularity of $c$. To overcome this issue, we first decouple the system, treat  solvability of the two equations separately, and finally obtain a joint solution with the help of Schauder's fixed point theorem.
In order to solve equation \cref{newc} for a fixed $v$ we adapt the aforementioned theory \cite{ACM021,ACM02,ACM05,ACM05Cauchy,ACMM06,ACM08,ACMM10}.

The rest of the paper is organised as follows. We begin with the preliminary \cref{prelim} where we collect some notation and useful facts. In particular, we recall 
some machinery which helps to operate with BV functions. In \cref{prset} we fully set up the problem and state the main result, the existence \cref{mainthm}.   \cref{Secc,Secv} are devoted to the solvability of \cref{newc,newv} for fixed $v$ and $c$, respectively. The most extensive Section of this work, \cref{Secc},  begins with \cref{Secaprio} where  some a priori bounds are derived for sufficiently regular solutions. It is succeeded by a study of some elliptic operators related to the $c$-equation in \cref{Secell}. Next, in \cref{Secdiscr},  we consider a family of regularised time discretisations of \cref{newc} for a fixed $v$. 
This allows us to pass to a  limit and thus obtain an entropy solution to \cref{newc} in \cref{Secsolc}. 
Finally, in \cref{SecSchauder} we establish the existence of  entropy solutions to the whole system \cref{new}.
\section{Preliminaries}\label{prelim}
Throughout this paper, $\Omega\subset \R^d$, $d\in\N$,  is a smooth bounded domain and $O\subset \R^N$, $N\in\N$, is any domain.

For the (continuous) dual of a normed space $X$ we use the notation $X^*$. Symbols $\rightharpoonup$ and $\overset{*}{\rightharpoonup}$ represent the weak and weak-$*$ convergence, respectively. 
By $<\cdot,\cdot>$ we denote the duality pairing.
\subsection{Spaces of scalar functions}

Let $k\in \N_0$.
As usual, $C^k(\overline{O})$ ($C(\overline{O}):=C^0(\overline{O})$)  denotes the space of $k$ times continuously differentiable real-valued functions in $\overline{O}$ equipped with their standard respective norms.
By $D(O)$ we denote the space of test functions in $O$. The closure of $D(O)$ in $C^k(\overline{O})$ ($C(\overline{O})$) is denoted by $C_0^k(\overline{O})$ ($C_0(\overline{O})$).
Similarly, we write, e.g. $C(E;F)$ for the space of continuous functions with domain $E$ and values in $F$.

We use standard Lebesgue spaces $L^p(\Omega)$, Sobolev spaces $W^{s,p}(\Omega)$ and  $H^s(\Omega):=W^{s,2}(\Omega)$ for various $p\in[1,\infty]$ and $s\in\R$.
We assume the reader to be familiar with their standard properties. 

Index $\nu$ added into the notation of a space of functions defined in $\overline{\Omega}$ means that this is the subspace of functions with zero normal derivative, understood in the trace sense where necessary, on the boundary  of $\Omega$.

\subsection{Measures and BV functions}\label{SecBV}
In this Subsection we introduce some notation and recall a few facts about Radon measures and BV functions. A detailed treatment of BV functions can be found, e.g. in \cite{AFP2000}. 
By ${\cal M}(\Omega)$ ($({\cal M}(\Omega))^d$) we denote the space of finite real (vector, with values in $\R^d$) Radon measures in $\Omega$. The norms in these spaces are given by the total variation. We follow a standard convention and  identify these measures with the corresponding continuous lineal functionals on $\Co$ defined via integration. We write $|\mu|$/$\mu^{ac}$/$\mu^s$ for the variation/absolutely continuous part/singular part of a Radon measure $\mu$, respectively. We also identify the absolutely continuous part with the corresponding density (with respect to the Lebesgue measure, that is). By $\frac{\mu_1}{\mu_2}$ we denote the Radon-Nikodym derivative of a measure $\mu_1$ with respect to a positive measure $\mu_2$.

By $|\cdot|$ we denote the $d$-dimensional Lebesgue measure and by ${\cal H}^{d-1}$ the $(d-1)$-dimensional Hausdorff measure in $\R^d$.

Unless specified otherwise, a.e. (a.a.), meaning almost everywhere (almost all), is to be understood  with respect to the Lebesgue measure of the dimension that is standardly used for a set in question.

As usual, $\BV$ stands for the subspace of $\LOne$ which consists of functions $w$ of bounded variation. The latter means that the distributional gradient $D w$ is given by a  Radon measure of a finite total variation. We write $\nabla w/D^sw/D^cw/D^jw$ for the absolutely continuous/singular/Cantor/jump parts of $D w$, respectively. If $w\in \WOne$, then, of course, $Dw=\nabla w$. We denote by $S_w$ the approximate discontinuity set and by  $J_w$ the set of approximative jump points of $w$. The approximate limit of $w$ at $x\in\Omega\backslash S_w$ is denoted by $\widetilde{w}(x)$. For $x\in J_w$, the approximate limits of $w$ arising in directions $\nu_w(x):=\frac{Dw}{|Dw|}(x)$ and $-\nu_w(x)$ are denoted by $w^+(x)$ and $w^-(x)$, respectively. 

As in \cite{ACM021,ACM02,ACM05,ACM05Cauchy,ACMM06,ACM08,ACMM10}, we need pairings defined via the 'partial integration formula', in our case for a bounded domain:
\begin{align}
\left<(Z,Dw),\rho\right>:=\int_{\Omega}-\rho w\nabla\cdot Z-Zw\cdot\nabla\rho\,dx\qquad \text{for all }\rho\in C_0(\overline{\Omega})\cap C^1(\overline{\Omega})\label{prod} 
\end{align}
for $w\in  \BV$ and $Z\in X_{\infty}(\Omega)$, where 
\begin{align}
X_{\infty}(\Omega):=\{Z\in (\LInf)^d:\ \nabla\cdot Z\in \LInf\}.  \label{Xinf}                                                                \end{align}
We make use of the following properties originating from  \cite{Anzellotti,KohnTem}.  
\begin{enumerate}
[label=(\arabic*),ref=(\arabic*)]
\item\label{pairingdef} The mapping $(\cdot,D\cdot):X_{\infty}(\Omega)\times \BV\rightarrow {\RM}$ is well-defined, bilinear, bounded, and satisfies
\begin{align}
 \|(Z,Dw)\|_{{\RM}}\leq \|z\|_{\LInf}\|D w\|_{{\RM}}\qquad\text{for all }Z\in X_{\infty}(\Omega)\text{ and }w\in  \BV.\nonumber
\end{align}
 \item Paring $(\cdot,\cdot)$ extends  product of functions  and  multiplication of a finite measure by a continuous function:
 \begin{alignat}{3}
 &(Z,Dw)=Z\cdot \nabla w&&\qquad\text{for all }Z\in X_{\infty}(\Omega),\ w\in \WOne, \nonumber\\
 &\left<(Z,Dw),\rho\right>=\left<Dw,Z\rho\right>&&\qquad\text{for all }Z\in (C(\overline{\Omega}))^d,\ \rho\in{ \Co}.\nonumber
\end{alignat}
\item Measure $(Z,Dw)$ is absolutely continuous with respect to $|Dw|$ and has with respect to it a bounded density.
\item For the absolutely continuous part of $(Z,Dw)$ we have the explicit formula 
\begin{subequations}
\begin{align}
(Z,Dw)^{ac}=Z\cdot\nabla w\qquad\text{a.e. in }\Omega. \label{ZDwa}                               \end{align}
\item
The singular part $(Z,Dw)^s=(Z,Dw)-Z\cdot\nabla w$ is a finite measure which is absolutely continuous with respect to $|D^s_xw|$ and satisfies
\begin{align}
 |(Z,Dw)^s|\leq \|Z\|_{\LInf}|D^sw|\qquad\text{in }\RM.\label{ZDws}
\end{align}
\end{subequations}

\end{enumerate}

\subsection{Spaces of functions and measures with values in Banach spaces}
In this Subsection, $I\subset\R_0^+$ is an interval, $p\in[1,\infty]$, and  $X$ is a Banach space with the dual space $X^*$. 

We denote: by $C(I;X)$ ($C^1(I;X)$) the space of continuous (continuously differentiable) functions $u:I\rightarrow X$, by $C_w(I;X^*)$ the space of weakly continuous functions $u:I\rightarrow X^*$, by $L^p(I;X)$ the space of strongly measurable functions $u:I\rightarrow X$ with $\|u\|_X\in L^p(I)$, and by $L^p_{w}(I;X)$ ($L^p_{w-*}(I;X^*)$) the space of weakly (weakly-$*$) measurable functions $u:I\to X$ ($u:I\rightarrow X^*$) with $\|u\|_{X}\in L^p(I)$ ($\|u\|_{X^*}\in L^p(I)$). 
Further, $W^{1,p}(I;X)$ stands for the Sobolev space of weakly differentiable functions $u\in L^p(I;X)$ such that $\partial_t u \in L^p(I;X)$. 

Let $Y$ be a space that contains $X$. If $u$ is a $Y$-valued distribution on $I$ and such that it can be extended to a function with values in $X$, then we identify $u$ with that extension. 

Since we only deal here with the cases where $X$ is a space of functions or measures, we typically  adopt  the notation $u(t,\cdot)$ rather than $u(t)$ for $t\in I$.

For a function $u:I\to X^d$, $d\in\N$, we often write $\|u(t,\cdot)\|_{X}$ instead of $\||u(t,\cdot)|\|_{X}$,  where $|\cdot|$ denotes the Euclidean norm in $\R^d$.

If $loc$ appears in the index of a space notation, then the corresponding defining property is assumed to hold on every compact subset of $I$.

If $f:I\to \RM$, then we set $(f^{ac})(t,\cdot):=(f(t,\cdot))^{ac}$ and $(f^s)(t,\cdot):=(f(t,\cdot))^s$ for $t\in I$. \\

By ${\cal M}(I ;X^*)$ we denote the space of $X^*$-valued countably additive regular vector measures in $I$ that are of bounded variation.  
 The reader is referred to \cite[Chapter 1  \S1]{DiestelUhl} for the definitions of the involved terms. We equip this vector measure space with the variation over $I$ as norm. An extension of the Riesz-Markov-Kakutani representation theorem holds for compact $I$\cite{Singer} (see also \cite[Chapter III \S19]{Dinculeanu}): 
 \begin{align}
  {\cal M}(I ;X^*)\cong(C(I ;X))^*.\label{RMKV}
 \end{align}
In view of \cref{RMKV} and a representation result given by \cref{LemRep} from \cref{AppVM} we identify $f\in L^1_{w-*}((0,T);X^*)$ with the vector measure defined by
\begin{align}
  \left<\mu_f(E),\varphi\right>:=\int_E\left<f(t),\varphi\right>dt\qquad\text{for all }\varphi\in X\text{ and Lebesgue measurable }E\subset I \nonumber%
 \end{align}
 as well as with the corresponding element of $(C(I;X))^*$.
\subsection{Interpolation and discretisation of vector-valued functions}\label{PrelimDI}
 For $n\in\N$ set
\begin{align*}
 \frac{1}{n}\N_0:=\left\{\frac{k}{n}:\ k\in\N_0\right\}.
\end{align*}
Let $X$ be some Banach space.  We use two standard ways of interpolating a function  $u:\frac{1}{n}\N_0\cap [0,T]\to X$: piecewise constant
\begin{alignat}{3}
 &\uPi_n{u}(t):=u\left(\frac{k-1}{n}\right)&&\qquad\text{in }\left(\frac{k-1}{n},\frac{k}{n}\right)\qquad\text{ for } k\in\{1,2,\dots,nT\},\nonumber\\
 &{\oPi_n}{u}(t):=u\left(\frac{k}{n}\right)&&\qquad\text{in }\left(\frac{k-1}{n},\frac{k}{n}\right)\qquad\text{ for } k\in\{1,2,\dots,nT\},\nonumber
\end{alignat}
and piecewise linear
\begin{align}
 \Lambda_n{u}(t):=&\frac{u\left(\frac{k}{n}\right)-u\left(\frac{k-1}{n}\right)}{\frac{1}{n}}\left(t-\frac{k}{n}\right)+u\left(\frac{k}{n}\right)\nonumber\\
 =&\frac{\frac{k}{n}-t}{\frac{1}{n}}u\left(\frac{k-1}{n}\right)+\frac{t-\frac{k-1}{n}}{\frac{1}{n}}u\left(\frac{k}{n}\right)\nonumber\\
 =&\frac{\frac{k}{n}-t}{\frac{1}{n}}\uPi_n{u}(t)+\frac{t-\frac{k-1}{n}}{\frac{1}{n}}{\oPi_n}{u}(t)\qquad\text{in }\left(\frac{k-1}{n},\frac{k}{n}\right)\qquad\text{ for }k\in\{1,2,\dots,nT\}.\label{convComb}
\end{align}

For $\widetilde u\in L^1_{loc}([0,T];X)$, its $n$-th discretisation via Steklov averages is given by
\begin{align}
 P_n\widetilde u\left(\frac{k}{n}\right):=n\int_{\frac{k-1}{n}}^{\frac{k}{n}}\widetilde u(t)\,dt\qquad\text{for } k\in\{1,2,\dots,nT\}.\nonumber
\end{align}
Further, we set
\begin{align*}
 P_n\widetilde u(0):=P_n\widetilde u\left(\frac{1}{n}\right).
\end{align*}
Clearly, the  introduced  interpolation and discretisation operators are linear and preserve non-negativity: for a functional space $X$
\begin{align}
 u\geq0\text{ in }\frac{1}{n}\N_0\cap[0,T]&\qquad \Rightarrow \qquad\uPi_nu,{\oPi_n}u,\Lambda_nu\geq0 \text{ in }[0,T],\nonumber\\
 \widetilde u\geq 0\text{ in }[0,T]&\qquad \Rightarrow \qquad P_n\widetilde u\geq0 \text{ in }\frac{1}{n}\N_0\cap[0,T].\nonumber
\end{align}
Several other useful properties of these operators are provided in \cref{AppendixInter}.
\subsection{Further notation}
We make use of functions
\begin{align*}
T_{m,M}(x):=&\begin{cases}m&\text{ for }x\leq m,\\ x&\text{ for }x\in(m,M),\\M&\text{ for }x\geq M,\end{cases}\qquad -\infty< m<M<\infty,\\
T_{\delta}:=&T_{-\delta,\delta},\\
 \sign_{\delta}(x):=&\begin{cases}-1&\text{ for }x<-\delta,\\\frac{1}{\delta}x&\text{ for }x\in[-\delta,\delta],\\1&\text{ for }x>\delta,\end{cases}\qquad \delta>0,\\
 x_+:=&\begin{cases}
       x&\text{ for }x\geq0,\\
       0&\text{ for }x<0.
      \end{cases}
\end{align*}

For $h:\R^k\times \R^d\to\R$,  $k,d\in\N$, we define its recession function with respect to the second variable as in  \cite{DalMaso}:
\begin{align}
 h^{\infty}(y,\xi):=\underset{s\to0}{\lim}sh\left(y,\frac{\xi}{s}\right)\qquad\text{for }y\in\R^k,\ \xi\in\R^d,\nonumber
\end{align}
{provided all limits exist.}

Finally, we make the following  conventions. \begin{itemize}
\item 
For indices $i\in\N$, $C_i$ denotes a positive 
constant or, alternatively, a positive non-decreasing function of its arguments. The dependence upon the parameters of the problem (see \cref{prset} below), i.e. the space dimension $d$, domain $\Omega$, constants $\chi$, $p_0$, and $\alpha_0$, functions $f_c$ and $f_v$, the involved norms of the initial data, etc. is mostly omitted.  
\item Further, $g:\R_0^+\to \R$ stands for an arbitrary locally Lipschitz function. We do not track dependence upon $g$ in the constants/functions $C_i$.
\end{itemize}

\section{Problem setting and main result}\label{prset}
Given a smooth bounded domain $\Omega\subset\R^d$, $d\in\N$, with the unit outer normal $\nu$,  we consider the system
\begin{subequations}\label{IBVPc}
 \begin{alignat}{3}
  &\partial_t c=\nabla\cdot\left(a(c,\nabla c)-\chi c \Phi(\nabla v)\right)+f_c(c,v)&&\qquad\text{in }(0,\infty)\times \Omega,\label{flc}\\
  &\left(a(c,\nabla c)-\chi c \Phi(\nabla v)\right)\cdot\nu=0&&\qquad\text{in }(0,\infty)\times \partial\Omega,\label{noflux}\\
  &c(0)=c_0&&\qquad\text{in } \{0\}\times \Omega,\label{inic}
 \end{alignat}
 \end{subequations}
 and
 \begin{subequations}\label{IBVPv}
 \begin{alignat}{3}
  &\partial_t v=D_v\Delta v+f_v(c,v)&&\qquad\text{in }(0,\infty)\times \Omega,\label{flv}\\
  &\partial_{\nu}v=0&&\qquad\text{in }(0,\infty)\times \partial\Omega,\label{bcv}\\
  &v(0)=v_0&&\qquad\text{in } \{0\}\times \Omega.\label{iniv}
 \end{alignat}
 \end{subequations}
We next list our assumptions on the parameters along with some of their properties:
\begin{align}
\chi\in(0,1),\qquad D_v>0;\label{asschi}
\end{align}
\begin{align}
 a(c,\xi):=\begin{cases}\frac{c\xi}{\sqrt{c^2+|\xi|^2}}&\text{for }c>0,\\
      0&\text{for }c=0, 
    \end{cases}\qquad\xi\in\R^d,\label{a}
\end{align}
so that   
\begin{subequations}\label{conda}
\begin{alignat}{3}
&a\in C(\R_0^+\times \R^d;\R^d),&&\\
&|a(c,\xi)|\leq c&&\qquad\text{for }c\geq0,\ \xi\in\R^d,\label{abound}\\
&|a(c_1,\xi)-a(c_2,\xi)|\leq |c_1-c_2|&&\qquad\text{for } c_1,c_2\geq0,\ \xi\in\R^d,\label{aLip1}\\
&a(c,\xi)\cdot\xi\geq c(|\xi|-c)&&\qquad\text{for } c\geq0,\ \xi\in\R^d,\label{acor}\\
&(a(c,\xi_1)-a(c,\xi_2))\cdot(\xi_1-\xi_2)\geq 0&&\qquad\text{for }c\geq0,\ \xi_1,\xi_2\in\R^d,\label{amon}\\
&(a(c,\cdot)\cdot(\cdot))^{\infty}(\xi)=c|\xi|&&\qquad\text{for }c\geq0,\ \xi\in\R^d;\label{recessa}
\end{alignat}
\end{subequations}
\begin{align}
\Phi(\xi):=\frac{\xi}{\sqrt{1+|\xi|^2}}\qquad\text{for } \xi\in\R^d,\label{Phi}
\end{align}
so that 
\begin{subequations}
\begin{alignat}{3}
&\Phi\in C^{\infty}_b(\R^d;\R^d),&& \nonumber\\
& |\Phi(\xi)|\leq 1&&\qquad\text{for } \xi\in\R^d;\label{Phi1}
\end{alignat}
\end{subequations}
\begin{subequations}\label{fc}
\begin{alignat}{3}
&f_c\in C^1(\R_0^+\times\R_0^+;\R),\label{fcfvreg}&&\\
 &f_c(c,v)\leq \Cl[F]{fc1}c&&\qquad\text{for }c,v\geq0,\label{asfc1} \\
 &f_c(0,v)=0&&\qquad\text{for }v\geq0,
\end{alignat}
\end{subequations}
where $\Cr{fc1}\in\R$;
\begin{subequations}\label{fv}
\begin{alignat}{3}
 &f_v\in C^1(\R_0^+\times\R_0^+;\R),\\
 &f_v(c,v)\leq \Cl[F]{fv1}(c)(v+\Cl[F]{fv2}(c))&&\qquad\text{for }c,v\geq0, \\
 &f_v(c,0)\geq0&&\qquad\text{for }c\geq0,
\end{alignat}
\end{subequations}
where $\Cr{fv1},\Cr{fv2}\in C(\R_0^+;\R_0^+)$;
\begin{align}
 &0\leq c_0\in \LInf,\label{c0bnd}\\
 &0\leq v_0\in W^{2\left(1-\frac{1}{p_0}\right),p_0}_{\nu} \qquad\text{for some }p_0>d+2.\label{v0bnd}
\end{align}
\begin{Remark}%
\begin{enumerate}
\item 
The specific choice of $a$ in  \cref{a} corresponds to $D_c=C=1$ in \cref{newc}. There is no loss of generality here because such an $a$ can be achieved through introducing rescaled time and space variables 
\begin{align*}
t':= \left(\frac{C}{D_c}\right)^{\frac{1}{2}}t, \qquad x':=\frac{C}{D_c}x.
\end{align*}

Moreover, our analysis is not tied to a FL diffusion such as in \cref{newc} and, similarly to, e.g.  \cite{ACM05Cauchy}, it  extends to more general FL diffusion terms. {This is in particular the case if $a$ satisfies \cref{conda} and there  exists a Lagrangian $L$ associated with it that satisfies conditions \cref{propL} below, but also in more general situations.} 
\item Less restrictive conditions on other parameters could also be considered. For instance, one could relax the one-sided sublinear growth condition \cref{fcfvreg} on $f_c$  or assume $\chi$ to be a function rather than a constant. 
To streamline the exposition, we abstained from such generalisations here. 

Note that the logistic growth $f_c=f_c(c)=c(1-c)$, such as previously adopted in \cite{ACM10}, is admissible under \cref{fc}.
\end{enumerate}
\end{Remark}

We define solutions to \cref{IBVPc}-\cref{IBVPv} as follows. 
\begin{Definition}[Weak and entropy solutions to \cref{IBVPc}]\label{Defsol1} 
Let $a$ and $\Phi$ be as defined in \cref{a} and \cref{Phi}, respectively, $f_c$ satisfy \cref{fc}, and  $c_0$ satisfy \cref{c0bnd}. We call a pair of functions $(c,v):\R_0^+\times \overline{\Omega}\rightarrow\R_0^+\times \R_0^+$  a \underline{global weak solution} to \cref{IBVPc}-\cref{IBVPv} if $c$ and $v$ satisfy:
\begin{subequations}\label{regc}
\begin{align}
&c\in L^{\infty}_{loc}(\R_0^+;\LInf) \cap L_{w,loc}^1(\R_0^+;\BV)\cap C_w(\R_0^+;\LInf),\label{ClassC_}\\
&c\in W^{1,\infty}_{loc}(\R_0^+,(\WOne)^*),
\end{align}
\end{subequations}
\begin{align}
 &v\in  C(\R_0^+;C^1_{\nu}(\overline{\Omega})),\label{classvinc}
\end{align}
  the weak formulation
 \begin{align}
  &\left<\partial_tc,\varphi\right>=\int_{\Omega}-J\cdot\nabla\varphi+f_c(c,v)\varphi\,dx\qquad \text{a.e. in }(0,\infty)\qquad\text{for }\varphi\in L^1_{loc}(\R_0^+;\WOne),\label{weakfc_}
 \end{align}
 where 
 \begin{align*}
  J:=a(c,\nabla c)-\chi c\Phi(\nabla v),
 \end{align*}
 and the initial condition  \cref{inic} in $\LInf$.
 
 We call $(c,v)$ a \underline{global entropy solution} to \cref{IBVPc}-\cref{IBVPv} if it is a global weak solution and additionally satisfies 
 for all $T\in\N$ and $0\leq m< M<\infty$:
\begin{align}
 \partial_t \int_0^cT_{m,M}(u)\,du+J\cdot \nabla T_{m,M}(c)-\nabla\cdot(JT_{m,M}(c))-f_c(c,v)T_{m,M}(c)\in L^1_{w-*,loc}(\R_0^+;\RM)\nonumber
\end{align}
 and the \underline{entropy inequalities}
  \begin{subequations}\label{entropyin}
  \begin{align}
   &-\Cl{CBNDchi}(T)|D^sT_{m,M}(c)|\nonumber\\
   \leq&\partial_t \int_0^{c}T_{m,M}(u)\,du+J\cdot \nabla T_{m,M}(c)-\nabla\cdot(JT_{m,M}(c))-f_c(c,v)T_{m,M}(c)\label{EstBelow}\\
   \leq &-\frac{1}{2}|D^sT_{m,M}^2(c)|+\frac{\chi}{2} D^sT_{m,M}^2(c)\cdot\Phi(\nabla v)\qquad\text{in }\RM\text{ a.e. in }(0,T),\label{EstAbove}
  \end{align}
  where
  \begin{align}
   \Cr{CBNDchi}(T):=\Cl{CBNDchi1}(1+T)e^{T\Cr{fc1}}
 \|c_0\|_{\LInf}
  \end{align}
for some constant $\Cr{CBNDchi1}$ that depends only on $d$, $\Omega$, and $\chi$.
\end{subequations}
\end{Definition}
\begin{Remark}
  \begin{enumerate}
\item The one-sided entropy inequalities \cref{EstAbove} are reminiscent of those in, e.g. \cite[Definition 6]{ACM05}. In contrast, inequalities  adopted in, e.g. \cite[Definition 4.4]{ACM05Cauchy} each contain  a pair of piecewise linear transformations of the form $T_{m,M}$. Such more delicate entropy inequalities, but also further  conditions on the time derivative of $c$ and/or boundary traces (compare \cite[Definition 6]{ACM05}) can be considered for \cref{IBVPc} as well. Since we do not aim at uniqueness, we chose to restrict ourselves to \cref{EstAbove}. 
\item In contrast to, e.g. \cite{ACM05}, an estimate from below,  \cref{EstBelow}, is included as well. 
Each weak solution to \cref{IBVPc}-\cref{IBVPv}  satisfies it, see \cref{RemLB}. 
\item {The entropy inequalities  \cref{EstBelow,EstAbove} ensure that the right-hand side of \cref{EstBelow} is squeezed  between two singular measures that are subordinate to the singular part of the  derivative of $T_{m,M}(c)$. This rather tight control on the evolution of cutoffs of $c$  relies heavily on the flux-limitation in \cref{flc}.  A meaningful solution concept involving more complex cutoffs was previously introduced for the original non-FL system \cref{KS} in  \cite{ZhigunSIMA}. However, in this case, terms that are quadratic rather than sublinear in the gradient of the $c$-component emerge in the approximation systems for those cutoffs. The issue of controlling the singular parts of such terms in the limit remains unresolved.
}
\end{enumerate}

\end{Remark}

{\begin{Remark}
  Choosing $f_v\equiv 0$ and $v_0$ a constant function  reduces  \cref{flc} to a FL (reaction)-diffusion equation for $c$ alone and allows a direct comparison with previous results.
  \begin{enumerate}
  \item  We expect the $c$-component of our solution to \cref{IBVPc} to be unique for fixed $v$ and, moreover, to coincide with the entropy solution in the sense of  \cite{ACM05} if $v$ is constant, see  \cref{RemMethod}.
   \item 
   Unlike the constructions in  \cite{ACM021,ACM02,ACM05,ACM05Cauchy,ACMM06,ACM08,ACMM10} and other works, where
 \begin{align*}
   &T_{{m},M}(c)\in L_{w,loc}^1(\R_0^+;\BV)\qquad \text{for }0<{m}<M<\infty,\\
   &c\in C(\R_0^+;\LOne),
  \end{align*}
  we have with \cref{ClassC_} that
  \begin{align*}
   &c\in L_{w,loc}^1(\R_0^+;\BV),\\
   &c\in C_w(\R_0^+;\LOne).
  \end{align*}
  This discrepancy is due to the differences in the derivation methods, see \cref{RemMethod}. Thus, our approach improves on the spatial regularity close to $\{c=0\}$, but maintains only weak continuity in time for $c$ taking values in $\LOne$. 
  Still, in view of the expected uniqueness (see previous item), $c$ then  enjoys the  regularity
  \begin{align*}
   c\in L_{w,loc}^1(\R_0^+;\BV)\cap C(\R_0^+;\LOne).
  \end{align*}
  \end{enumerate}
  
\end{Remark}
}

\begin{Definition}[Strong solution to \cref{IBVPv}]\label{Defsol2} 
Let $f_v$ satisfy \cref{fv} and $v_0$ satisfy \cref{v0bnd}. We call a pair of functions $(c,v):\R_0^+\times \overline{\Omega}\rightarrow\R_0^+\times \R_0^+$  a \underline{global strong solution} to \cref{IBVPv} if $c$ and $v$ have the regularity 
\begin{align}
&c\in L^{\infty}_{loc}(\R_0^+;\LInf)\label{regcinv}
\end{align}
and
\begin{align}
 &v\in L_{loc}^{p_0}(\R_0^+;W^{2,p_0}_{\nu}(\Omega))\cap W^{1,p_0}_{loc}(\R_0^+;L^{p_0}(\Omega)),\label{regv}
\end{align}
respectively, and satisfy \cref{flv} strongly in $ L_{loc}^{p_0}(\R_0^+;L^{p_0}(\Omega))$ and  \cref{bcv,iniv} pointwise.                                                                                                                                                              
  
\end{Definition}
\begin{Remark}
 Combining the trace theorem \cite[Chapter 2 Theorem 4.10.2]{Amannbook95} and the  standard results on interpolation and embedding of Sobolev spaces and using assumption $p_0>d+2$, we have the continuous embeddings
  \begin{subequations}\label{embeddings}
\begin{align}
 W^{1,p_0}((0,T);L^{p_0}(\Omega))\cap L^{p_0}((0,T);W^{2,p_0}_{\nu}(\Omega))\subset &C\left([0,T];W^{2\left(1-\frac{1}{p_0}\right),p_0}_{\nu}\right)\\ \subset & C([0,T];C^1_{\nu}(\overline\Omega)).
\end{align} 
\end{subequations}
Thus, the $v$-component of a strong solution to \cref{IBVPv} satisfies
\begin{align}
 v\in C\left([0,T];W^{2\left(1-\frac{1}{p_0}\right),p_0}_{\nu}\right)\subset  C([0,T];C^1_{\nu}(\overline\Omega)),\label{contiv}
\end{align}
so that \cref{bcv,iniv} can indeed be stated pointwise.

\end{Remark}
\begin{Definition}[Solution to \cref{IBVPc}-\cref{IBVPv}]\label{Defsol}
 We call a pair of functions $(c,v)$ a \underline{global solution} to \cref{IBVPc}-\cref{IBVPv} if it is a global entropy solution to \cref{IBVPc} and a global strong solution to \cref{IBVPv}.
\end{Definition}

Our main result now reads:
\begin{Theorem}[Existence for \cref{IBVPc}-\cref{IBVPv}]\label{mainthm}
 Let assumptions \cref{asschi,a,Phi,fc,fv,c0bnd,v0bnd} be satisfied.  Then there exists a pair of functions $(c,v)$ 
that solves \cref{IBVPc}-\cref{IBVPv} in the sense of \cref{Defsol}.
\end{Theorem} 
The proof of \cref{mainthm} requires some preparation and is therefore postponed to the final  \cref{SecSchauder}. It is based on solving  \cref{IBVPc,IBVPv} for fixed $v$ and $c$, respectively, and then obtaining a joint solution to the whole system  \cref{IBVPc}-\cref{IBVPv} by means of Schauder's fixed point theorem. Smoothness of $v$ allows to establish the solvability of \cref{IBVPc} in the entropy sense using an approach that resembles the one developed and pursued in \cite{ACM021,ACM02,ACM05,ACM05Cauchy,ACMM06,ACM08,ACMM10}.
\begin{Remark}
\begin{enumerate}
 \item Our existence proof is valid only for $\chi\in(0,1)$. For $\chi>1$, the study in  \cite{BelWinkler20172} of a parabolic-elliptic version of \cref{IBVPc}-\cref{IBVPv} in the radial symmetric case  indicated that a local solution may cease to exist globally due to a blow-up in finite time, and that for any initial mass $m_0:=\int_{\Omega}c_0\,dx$ in higher dimensions, and, for sufficiently large $m_0$, even in the case of dimension one.
 \item Due to a possibly  low regularity of the $c$-component and the strong coupling of the equations the uniqueness of solutions remains a challenging open issue.
\end{enumerate}
\end{Remark}

\section{Existence of entropy solutions to \texorpdfstring{\cref{IBVPc}}{} for fixed \texorpdfstring{$v$}{v}}\label{Secc}
In this Section we establish the existence of entropy solutions to \cref{IBVPc} for given $v$. To start, we derive some a priori bounds for the $c$-component in \cref{Secaprio}. Next, we study several spatial differential operators that arise in the context of \cref{flc}-\cref{noflux} in \cref{Secell}. These steps serve as  preparation for the proof of a result on well-posedness of a regularised time-discrete approximation of \cref{IBVPc} that we give in \cref{Secdiscr}. In \cref{Secsolc} we are then finally able to deduce the existence of entropy solutions to \cref{IBVPc} that depend on $v$ in a stable fashion.
\subsection{A priori bounds for \texorpdfstring{$c$}{c}}\label{Secaprio}
This Subsection is devoted to the  derivation of a priori estimates for a  weak solution $c$ of \cref{IBVPc} under the additional regularity assumption 
\begin{align}
c\in L^1_{loc}(\R_0^+;\WOne).   \label{addregc}                                                                                                                                                              \end{align}
In the proofs, estimates that we obtain for time-dependent quantities are to be understood to hold a.e., at least. 
\begin{Lemma}%
\label{Lemcbnd} 
Let $\chi\in(0,1)$ and $(c,v)$ be a weak solution to \cref{IBVPc}
such that \cref{addregc} is satisfied. 
Then for all $T>0$
\begin{align}
\|c\|_{L^{\infty}((0,T);\LInf)}
\leq \Cl{CBND_1}(1+T)e^{T\Cr{fc1}}
 \|c_0\|_{\LInf}=:\Cl{CBND}(T),\label{cbound}
\end{align}
where $\Cr{CBND_1}$  depends only on $d$, $\Omega$, and $\chi$.
\end{Lemma}
\begin{proof}
Throughout this proof, constants $C_i$ depend  on $d$, $\Omega$, and $\chi$ alone. 

We use the standard method of propagation of $L^p$ bounds. To begin with, taking $\varphi\equiv1$ in \cref{weakfc_} and using \cref{asfc1} and the Gronwall lemma one readily verifies 
\begin{align}
 \|c(t,\cdot)\|_{\LOne}\leq e^{t\Cr{fc1}}\|c_0\|_{\LOne}\qquad\text{for all }t\geq0.\label{cL1}
\end{align}
In particular, if $c_0=0$ a.e., then by \cref{cL1} we have $c=0$ a.e. This proves \cref{cbound} for this case. Therefore, from now on we assume
\begin{align*}
 |\{c_0>0\}|>0.
\end{align*}
In \cref{weakfc_}, choose $$\varphi:={p} c^{{p}-1},\ {p}\geq2.$$
Due to \cref{ClassC_} and \cref{addregc} we have $\varphi\in L^1_{loc}(\R_0^+;\WOne)$, as needed for \cref{weakfc_}, and so by the weak chain rule
\begin{align}
\frac{d}{dt}\int_{\Omega}c^{p}dx
=&\int_{\Omega}-{p}({p}-1)\left(a(c,\nabla c)-\chi c \Phi(\nabla v)\right)\cdot  c^{{p}-2}\nabla c+{p} c^{{p}-1}f_c(c,v)\,dx.\label{est00}
\end{align}
Using \cref{addregc}, the essential boundedness of $c$ and $v$, and properties of $a$, $\Phi$, and $f_c$, we conclude from \cref{est00} that
\begin{align}
 \int_{\Omega}c^{p}dx\in W^{1,\infty}_{loc}(\R_0^+)\subset C(\R_0^+).\nonumber
\end{align}
Further, %
assumption $\chi\in(0,1)$, properties of $a$ and $\Phi$, and the weak chain rule together allow to estimate the right-hand side of \cref{est00} to obtain
\begin{align}
\frac{d}{dt}\|c^{p}\|_{\LOne}\leq&\int_{\Omega}-{p}({p}-1)\left((1-\chi)|\nabla c|-c\right) c^{{p}-1}+\Cr{fc1}{p} c^{p}\,dx\nonumber\\
 = &-(1-\chi)({p}-1)\left\|\nabla c^{p}\right\|_{\LOne}+{p}({p}-1)\|c^{p}\|_{\LOne}+\Cr{fc1}p\|c^{p}\|_{\LOne}\nonumber\\
 \leq&-\Cl{Cd1}p\left\|\nabla c^{p}\right\|_{\LOne}+\Cl{Cd2}p^2\|c^{p}\|_{\LOne}+\Cr{fc1}p\|c^{p}\|_{\LOne}.\label{est5}
\end{align}
Next, we combine a version of the Gagliardo-Nirenberg interpolation  inequality in the form provided by \cite[Lemma 2.3]{LiLankeit}, choosing there $r:=p:=1$ and $q:=s:=1/2$, with Young's inequality to obtain
\begin{align}
 \|w\|_{\LOne}\leq&\varepsilon\|\nabla w\|_{\LOne}+\Cl{Cint}\left(1+\varepsilon^{-\frac{\theta}{1-\theta}}\right)\left\|w^{\frac{1}{2}}\right\|_{\LOne}^{2}\qquad\text{for all }\varepsilon>0.\label{interp}
 \end{align}
 For $$w:= c^{p}\qquad\text{and}\qquad\varepsilon:=\frac{\Cr{Cd1}}{2\Cr{Cd2}}p^{-1}$$ inequality \cref{interp} leads to
\begin{align}
 \|c^{p}\|_{\LOne}
 \leq &\frac{\Cr{Cd1}}{2\Cr{Cd2}}p^{-1}\left\| \nabla c^{p}\right\|_{\LOne}+\Cr{Cint}\left(1+\left(\frac{\Cr{Cd1}}{2\Cr{Cd2}}\right)^{-\frac{\theta}{1-\theta}}{p}^{\frac{\theta}{1-\theta}}\right)\left\|c^{\frac{p}{2}}\right\|_{\LOne}^2\nonumber\\
 \leq &\frac{\Cr{Cd1}}{2\Cr{Cd2}}p^{-1}\left\| \nabla c^{p}\right\|_{\LOne}+\C{p}^{\frac{\theta}{1-\theta}}\left\|c^{\frac{p}{2}}\right\|_{\LOne}^2.\label{interp1}
\end{align}
Combining \cref{est5,interp1}, we conclude that
\begin{align}
\frac{d}{dt}\|c^{p}\|_{\LOne}
 \leq&\Cl{Cint1}{p}^{\Cl{ckappa}}\left\|c^{\frac{p}{2}}\right\|_{\LOne}
 ^2+\Cr{fc1}p\|c^{p}\|_{\LOne}.\label{est10_}
\end{align}
Applying Gronwall's lemma to \cref{est10_} and multiplying throughout by $e^{-t\Cr{fc1}p}$ yields
\begin{align}
 \left\|\left(e^{-t\Cr{fc1}}c(t,\cdot)\right)^{p}\right\|_{\LOne}\leq \|c_0^{p}\|_{\LOne}+\Cr{Cint1}{p}^{\Cr{ckappa}}\int_0^t\left\|\left(e^{-s\Cr{fc1}}c(s,\cdot)\right)^{\frac{p}{2}}\right\|_{\LOne}^2\,ds\qquad\text{for }t\geq0,\ p\geq2. \label{est10}
\end{align}
Let \begin{align}
 &B_{p}(t):=\frac{\left\|\left(e^{-t\Cr{fc1}}c(t,\cdot)\right)^{p}\right\|_{\LOne}}{\|c_0^p\|_{\LOne}}
 \qquad\text{for }t\geq0,\ p\geq1.\label{Bp}
\end{align}
Obviously,
\begin{align}
 \|B_{p}\|_{L^{\infty}((0,T))}\geq B_p(0)=1\qquad\text{for all }T>0,\ p\geq 1.\label{Blone}
\end{align}
Due to \cref{cL1} is also holds that 
\begin{align}
 &\|B_1\|_{L^{\infty}((0,\infty))}= B_1(0)=1.\label{bndB1}
\end{align}
Combining \cref{est10,Bp,Blone} and H\"older's inequality, we have for $p\geq2$
\begin{align}
 \|B_p\|_{L^{\infty}(0,T)}\leq&1+T\frac{\left\|c_0^{\frac{p}{2}}\right\|_{\LOne}^2}{\|c_0^p\|_{\LOne}} \Cr{Cint1}{p}^{\Cr{ckappa}}\left\|B_{\frac{p}{2}}\right\|_{L^{\infty}(0,T)}^2\nonumber\\
 \leq& 1+T|\Omega|\Cr{Cint1}{p}^{\Cr{ckappa}}\left\|B_{\frac{p}{2}}\right\|_{L^{\infty}(0,T)}^2\nonumber\\
 \leq&(1+T\Cl{C57}){p}^{\Cr{ckappa}}\left\|B_{\frac{p}{2}}\right\|_{L^{\infty}(0,T)}^2.\label{rec1}
\end{align}
Choosing ${p}:=2^{n+1}$, $n\in\N_0$, in \cref{rec1}, we arrive at the sequence of recursive inequalities
\begin{align}
 &\left\|B_{2^{n+1}}\right\|_{L^{\infty}((0,T))}\leq (1+T\Cr{C57})\left(2^{\Cr{ckappa}}\right)^n \left\|B_{2^n}\right\|_{L^{\infty}((0,T))}^2 \qquad \text{for }n\in\N_0.\label{recyn}
\end{align}
Combining \cref{recyn,bndB1} with  \cite[Chapter II Lemma 5.6]{LSU}, we deduce 
\begin{align}
 \frac{e^{-t\Cr{fc1}}
\|c(t,\cdot)\|_{\LInf}
 }{\|c_0\|_{\LInf}}
 \leq &\frac{\underset{s\in[0,T]}{\max} e^{-s\Cr{fc1}}\|c(s,\cdot)\|_{\LInf}}{\|c_0\|_{\LInf}}\nonumber\\
 =&\underset{n\to\infty}{\lim}\left\|B_{2^n}\right\|_{L^{\infty}((0,T))}^{2^{-n}}\nonumber\\
 \leq& \underset{n\to\infty}{\lim} ((1+T\Cr{C57}))^{1-2^{-n}}2^{\Cr{ckappa}(1-2^{-n}(n+1))}\left\|B_1\right\|_{L^{\infty}((0,T))}\nonumber\\
 =&(1+T\Cr{C57})2^{\Cr{ckappa}}\nonumber\\
 \leq& (1+T)\Cr{CBND_1}\qquad\text{for all }0\leq t\leq T,\nonumber
\end{align}
yielding \cref{cbound}.
\end{proof}
\begin{Lemma}%
\label{Lemnablac} Under the assumptions of \cref{Lemcbnd}  it holds for all $T>0$ that
\begin{align}
 \|\nabla c\|_{L^1((0,T);\LOne)}\leq &\Cl{Cc0}+T\Cl{C111}\left(\Cr{CBND}(T),\|v\|_{L^{\infty}((0,T);\LInf)}\right)\nonumber\\
 =:&\Cl{CD}\left(T,\Cr{CBND}(T),\|v\|_{L^{\infty}((0,T);\LInf)}\right).\label{DCC}
\end{align}
\end{Lemma}
\begin{proof}
Set $$\varphi:=\ln(\max\{c,{m}\}),\ {m}\in(0,1).$$
Due to \cref{ClassC_} and \cref{addregc} we have $\varphi\in L^1_{loc}(\R_0^+;\WOne)$, as needed in \cref{weakfc_}. Plugging $\varphi$ into \cref{weakfc_}, integrating over $(0,T)$, and using the weak chain rule,  
properties of $a,\Phi$, and $f_c$, estimate \cref{cbound}, and boundedness of $v$, we  obtain  
\begin{align}
&\int_{\Omega}\int_{c_0}^{c(T,\cdot)}\ln(\max\{s,{m}\})\,dsdx\nonumber\\
=&\int_0^T\int_{\{c>{m}\}}-\left(a(c,\nabla c)-\chi  c\Phi(\nabla v)\right)\cdot \nabla \ln c\,dx+\int_{\Omega} f_c(c,v)\ln(\max\{c,{m}\})\,dxdt\nonumber\\
\leq&\int_0^T\int_{\{c>{m}\}}-(1-\chi)\left|\nabla c\right|+c\,dxdt+\int_0^T\int_{\Omega} \left|\frac{f_c(c,v)-f_c(0,v)}{c}\right|\max\{c,{m}\}|\ln(\max\{c,{m}\})|\,dxdt\nonumber\\
\leq& -(1-\chi)\int_0^T\int_{\{c>{m}\}}\left|\nabla c\right|\,dxdt+T\Cl{C111_}\left(\|c\|_{L^{\infty}((0,T);\LInf)},\|v\|_{L^{\infty}((0,T);\LInf)}\right).\label{est1_}
\end{align}
Thanks to the monotone convergence theorem and the fact that  $\nabla w=0$ a.e. in $w=0$ for any $w\in \WOne$ we have
\begin{align}
 \int_0^T\int_{\{c>{m}\}}\left|\nabla c\right|\,dxdt\underset{{m}\to0}{\to}\int_0^T\int_{\Omega}\left|\nabla c\right|\,dxdt.\label{lim1}
\end{align}
Further, since $\int_0^c\ln(\max\{s,{m}\})\,ds-\int_0^c\ln s\,ds$ converges to zero uniformly on finite intervals and $c(T,\cdot)\in \LInf$ due to \cref{ClassC_}, 
\begin{align}
 \int_{\Omega}\int_{c_0}^{c(T,\cdot)}\ln(\max\{s,{m}\})\,dsdx\underset{{m}\to 0}{\to}\int_{\Omega}(s\ln s-s)|_{s=c_0}^{c(T,\cdot)}dx.\label{lim2}
\end{align}
Passing to the limit in  \cref{est1_} and using \cref{lim1,lim2}, we obtain
\begin{align}
 \int_{\Omega}(s\ln s-s)|_{s=c_0}^{c(T,\cdot)}dx
 \leq& -(1-\chi)\int_0^T\int_{\Omega}\left|\nabla c\right|\,dxdt+T\Cr{C111_}\left(\|c\|_{L^{\infty}((0,T);\LInf)},\|v\|_{L^{\infty}((0,T);\LInf)}\right).\label{est1_2_1}
\end{align}
Rearranging \cref{est1_2_1} and using $s\ln s-s\geq -1$ for all $s>0$, 
we arrive at the estimate
\begin{align}
 (1-\chi)\int_0^T\int_{\Omega}\left|\nabla c\right|\,dxdt
 \leq& 1+\int_{\Omega}c_0\ln c_0-c_0\,dx+T\Cr{C111_}\left(\|c\|_{L^{\infty}((0,T);\LInf)},\|v\|_{L^{\infty}((0,T);\LInf)}\right)\nonumber\\
 \leq&\C+T\Cr{C111_}\left(\Cr{CBND}(T),\|v\|_{L^{\infty}((0,T);\LInf)}\right).\label{est1_2}
\end{align}
Together,  \cref{cbound,est1_2}  and  assumption $\chi<1$ yield \cref{DCC}.
\end{proof}
\subsection{Auxiliary operators}\label{Secell}
In preparation for the analysis of the IBVP \cref{IBVPc}, we introduce and study several operators defined on spaces of functions which depend on $x$ alone. We start with a family of operators $\A[v]$ parametrised by $v$ and given by 
\begin{align}
 \left<\A[v](c),\varphi\right>:=
 &-\left<a(|c|,\nabla c)-\chi c\Phi(\nabla v),\nabla\varphi\right>.\label{OperA1}
\end{align}
The proof of the following lemma is straightforward, so we omit it.
\begin{Lemma}%
 Let $a$ and $\Phi$ be as defined in \cref{a} and \cref{Phi}, respectively.
 For all $v\in C^1(\overline{\Omega})$ the corresponding operator $\A[v]$ is well-defined, continuous, and bounded between
 \begin{enumerate}
 [label=(\arabic*),ref=(\arabic*)]
  \item $\HOne$ and $(\HOne)^*$,
  \item $\WOne\cap \LInf$ and $(\WOne)^*$.
 \end{enumerate}
\end{Lemma}

Our next Lemma establishes a one-sided stability property of $\A[\cdot]$. 
\begin{Lemma}\label{Kato1}
Let 
\begin{align*}
&0\leq c_1,c_2\in \WOne\cap \LInf,\\
  &v_1,v_2\in W^{2,1}(\Omega)\cap C^1_{\nu}(\overline\Omega).
\end{align*}
 Then the following estimate holds: 
 \begin{align}
  &\underset{\delta\rightarrow0}{\lim\sup}\left<\A[v_1](c_1)-\A[v_2](c_2), \sign_{\delta}(c_1-c_2)\right>\nonumber\\
  \leq &\Cr{C19}\left(\|\nabla c_2\|_{\LOne}\|v_1-v_2\|_{C^1(\overline\Omega)}+\|c_2\|_{\LInf}\|v_1-v_2\|_{W^{2,1}(\Omega)}\right).\label{Kato2}
 \end{align}
\end{Lemma}
\begin{proof}
For any $v\in W^{2,1}_{\nu}(\Omega)$ and $c,\varphi\in \WOne\cap \LInf$ we can integrate by parts in the second summand in \cref{OperA1} to obtain 
 \begin{align}
 \left<\A[v](c),\varphi\right>=
 &-\int_{\Omega}a({|c|},\nabla c)\cdot\nabla\varphi\,dx-\chi\int_{\Omega} \nabla\cdot(c\Phi(\nabla v))\varphi\,dx.\label{OperA}                                                                                                                                                                       \end{align}
Since $c_1,c_2\in \WOne$, we have
\begin{align}
\varphi:=\sign_{\delta}(c_1-c_2)\in \WOne\cap \LInf\label{varphi}                                                                            \end{align}
for every $\delta>0$.
Subtracting  \cref{OperA} for $(c,v)=(c_1,v_1)$ from \cref{OperA} for  $(c,v)=(c_2,v_2)$ and testing with $\varphi$ as defined in \cref{varphi}, we obtain using 
properties of $a$, $\Phi$, and $\sign_{\delta}$,
the weak chain rule, and H\"older's inequality that
\begin{align}
&\left<\A[v_1](c_1)-\A[v_2](c_2), \sign_{\delta}(c_1-c_2)\right>\nonumber\\
=&-\int_{\Omega}(a(c_1,\nabla c_1)-a(c_2,\nabla c_2))\cdot\nabla \sign_{\delta}(c_1-c_2)-\int_{\Omega}\nabla\cdot\left(c_1 \Phi(\nabla v_1)- c_2\Phi(\nabla v_2)\right) \sign_{\delta}(c_1-c_2)\,dx\nonumber\\
=&-\int_{\{0<|c_1-c_2|<\delta\}}(a(c_1,\nabla c_1)-a(c_1,\nabla c_2))\cdot \frac{1}{\delta}(\nabla c_1-\nabla c_2)\,dx\nonumber\\
&-\int_{\{0<|c_1-c_2|<\delta\}}(a(c_1,\nabla c_2)-a(c_2,\nabla c_2))\cdot \frac{1}{\delta}\nabla (c_1-c_2)\,dx\nonumber\\
&-\int_{\Omega}\nabla\cdot((c_1-c_2) \Phi(\nabla v_1)) \sign_{\delta}(c_1-c_2)\,dx-\int_{\Omega}\nabla\cdot\left(c_2( \Phi(\nabla v_1)- \Phi(\nabla v_2))\right) \sign_{\delta}(c_1-c_2)\,dx\nonumber\\
\leq&\int_{\{0<|c_1-c_2|<\delta\}}|a(c_1,\nabla c_2)-a(c_2,\nabla c_2)| \frac{1}{\delta}|\nabla (c_1-c_2)|\,dx\nonumber\\
&-\int_{\Omega}\nabla\cdot((c_1-c_2) \Phi(\nabla v_1)) \sign_{\delta}(c_1-c_2)\,dx+\int_{\Omega}|\nabla\cdot\left(c_2( \Phi(\nabla v_1)- \Phi(\nabla v_2))\right)|\,dx\nonumber\\
\leq&\int_{\{0<|c_1-c_2|<\delta\}}|\nabla(c_1-c_2)|\,dx-\int_{\Omega}\nabla\cdot((c_1-c_2) \Phi(\nabla v_1)) \sign_{\delta}(c_1-c_2)\,dx\nonumber\\
&+\Cl{C19}\left(\|\nabla c_2\|_{\LOne}\|v_1-v_2\|_{C^1(\overline\Omega)}+\|c_2\|_{\LInf}\|v_1-v_2\|_{W^{2,1}(\Omega)}\right).\label{est8}
\end{align}
Passing to the limit in \cref{est8} as $\delta\rightarrow0$ and using the dominated convergence theorem, the weak chain rule, and the partial integration formula combined with $\partial_{\nu}v_1=0$, so that $\Phi(\nabla v)\cdot\nu=0$ on $\partial\Omega$, we arrive at
\begin{align}
 &\underset{\delta\rightarrow0}{\lim\sup}\left<\A[v_1](c_1)-\A[v_2](c_2), \sign_{\delta}(c_1-c_2)\right>\nonumber\\
  \leq 
  &-\int_{\Omega}\nabla\cdot((c_1-c_2) \Phi(\nabla v_1))\sign(c_1-c_2)\,dx\nonumber\\
  &+\Cr{C19}\left(\|\nabla c_2\|_{\LOne}\|v_1-v_2\|_{C^1(\overline\Omega)}+\|c_2\|_{\LInf}\|v_1-v_2\|_{W^{2,1}(\Omega)}\right)\nonumber\\
 = 
  &-\int_{\Omega}\nabla\cdot(|c_1-c_2| \Phi(\nabla v_1))\,dx\nonumber\\
  &+\Cr{C19}\left(\|\nabla c_2\|_{\LOne}\|v_1-v_2\|_{C^1(\overline\Omega)}+\|c_2\|_{\LInf}\|v_1-v_2\|_{W^{2,1}(\Omega)}\right)\nonumber\\
  =&\Cr{C19}\left(\|\nabla c_2\|_{\LOne}\|v_1-v_2\|_{C^1(\overline\Omega)}+\|c_2\|_{\LInf}\|v_1-v_2\|_{W^{2,1}(\Omega)}\right).\nonumber
\end{align}
\end{proof}
A direct consequence of \cref{Kato1} for $v_1=v_2=:v$ is the following Corollary.
\begin{Corollary}[Inequality of Kato type]
Let  
\begin{align}
  &0\leq c_1,c_2\in \WOne\cap \LInf,\nonumber\\
  &v\in W^{2,1}(\Omega)\cap C^1_{\nu}(\overline\Omega).\nonumber
  \end{align}
  Then it holds that
 \begin{align}
  \underset{\delta\rightarrow0}{\lim\sup}\left<\A[v](c_1)-\A[v](c_2), \sign_{\delta}(c_1-c_2)\right>
  \leq &0.\nonumber%
 \end{align}
\end{Corollary}

Next, we introduce regularised  perturbations of $\A[v]$ that include a source term.
For $m,n\in\N$ and $L>0$ let
\begin{subequations}
\begin{align}
\left<\B_m[v](c),\varphi\right>:=&\left<\A[v](c),\varphi\right>+\int_{\Omega}-\frac{1}{m}\nabla c\cdot\nabla \varphi+f_c(c,v)\varphi\,dx,\label{Bm}\\
 \left<\B_{L,m}[v](c),\varphi\right>:=&\left<\A[v](c),\varphi\right>+\int_{\Omega}-\frac{1}{m}\nabla c\cdot\nabla \varphi+f_c(T_{0,L}(c),v)\varphi\,dx,\label{DefBmn}\\
 \left<\B_{n,L,m}[v](c),\varphi\right>:=&\left<\B_{L,m}[v](c),\varphi\right>-n\int_{\Omega}c\varphi\,dx.\nonumber
\end{align}
\end{subequations}
In \cref{DefBmn}, the truncation by means of $T_{0,L}$ serves to turn the source term into a Lipschitz continuous mapping.
\begin{Lemma}\label{PropB}
 Let $a$ and $\Phi$ be as defined in \cref{a} and \cref{Phi}, respectively, $f_c$ satisfy \cref{fc}, $m,n\in\N$, $L>0$, and $v\in C^1(\overline{\Omega})$. Then  operators $$\B_{L,m}[v],\B_{n,L,m}[v]:\HOne\rightarrow (\HOne)^*$$ are well-defined, continuous, and bounded.
If
$$n>(\Cr{fc1})_++1$$
then $-\B_{n,L,m}[v]$ is  pseudo-monotone and coercive.
\end{Lemma}
\begin{proof}
Given that $\A[v]$ is well-defined, continuous, and bounded and $f_c(T_{0,L}(\cdot),\cdot)$ is Lipschitz, it is evident that $\B_{L,m}[v]$ and $\B_{n,L,m}[v]$ are well-defined, continuous, and bounded.  
We turn to the remaining properties of $-\B_{n,L,m}[v]$. Set
\begin{align*}
 A_1(x,c,\xi):=&\frac{1}{m}\xi+a(|c|,\xi)-\chi c\Phi(\nabla v(x))\\
 A_0(x,c):=&nc-f_c(T_{0,L}(c),v(x)).
\end{align*}
Using assumption $\chi<1$ and properties of $v$, $a$, $\Phi$, and $f_c$, we obtain 
\begin{subequations}\label{A10conti}
\begin{align}
&A_1\in C(\overline{\Omega}\times\R\times\R^d;\R^d),\\
&A_0\in C(\overline{\Omega}\times\R);
\end{align}
\end{subequations}
for all $x\in\Omega$, $c\in\R$, and $\xi\in\R^d$ 
\begin{subequations}\label{A1bnd}
\begin{align}
 |A_1(x,c,\xi)|\leq &\frac{1}{m}\left|\xi\right|+(1+\chi)|c|,\\
 |A_0(x,c)|\leq& n|c|+\Cl{fcL1}\left(L,\|v\|_{\LInf}\right);
\end{align}
for all $x\in\Omega$, $c\in\R$, and $\xi_1,\xi_2\in\R^d$, $\xi_1\neq\xi_2$,  
\end{subequations}
\begin{align}
 \left(A_1(x,c,\xi_1)-A_1(x,c,\xi_2)\right)\cdot (\xi_1-\xi_2)\geq \frac{1}{m}|\xi_1-\xi_2|^2>0;\label{Apm}
\end{align}
and for $x\in\Omega$, $c\in\R$, and $\xi\in\R^d$
\begin{align}
 A_1(x,c,\xi)\cdot\xi+A_0(x,c)c\geq&\frac{1}{m}|\xi|^2 +|c|((1-\chi)|\xi|-|c|)+(n-\Cr{fc1})c^2\nonumber\\
\geq&\frac{1}{m}|\xi|^2+(n-1-\Cr{fc1})c^2.\label{Acoerc}
\end{align}
Observe that 
\begin{align}
 -\left<\B_{n,L,m}[v](c),\varphi\right>=\int_{\Omega} A_1(\cdot,c,\nabla c)\cdot\nabla\varphi+A_0(\cdot,c)\varphi\,dx\qquad \text{for all }c,\varphi\in \HOne.\nonumber
\end{align}
Let $n>(\Cr{fc1})_++1$. Due to properties \cref{A10conti,A1bnd,Acoerc,Apm} operator $-\B_{n,L,m}[v](c)$ satisfies the assumptions of \cite[Theorem 1]{Browder} for $m=1$, $p=2$, so that it is pseudo-monotone by that result. Finally, coercivity is the direct consequence of \cref{Acoerc}.

\end{proof}

\subsection{Well-posedness of time-discrete approximations of   \texorpdfstring{\cref{IBVPc}}{} }\label{Secdiscr}

In preparation for the construction of solutions to \cref{IBVPc} for a fixed $v$ we study  a family of local time-discrete approximations of this problem. Let $n,m\in\N$. We introduce system 
\begin{subequations}\label{ellpr}
\begin{align}
&c_{nm}(0,\cdot)=c_0,\\
 &\frac{c_{nm}\left(\frac{k}{n},\cdot\right)-c_{nm}\left(\frac{k-1}{n},\cdot\right)}{\frac{1}{n}}=\B_m\left[P_nv\left(\frac{k}{n},\cdot\right)\right]\left(c_{nm}\left(\frac{k}{n},\cdot\right)\right)\qquad\text{for }k\in\N,\label{appr1e}
\end{align}
\end{subequations}
with $\B_m$ defined in \cref{Bm}.
In the notation from \cref{PrelimDI} equations \cref{appr1e} can also be rewritten in the form of a single equation involving interpolations,
\begin{align}
 \partial_t \Lambda_n{c}_{nm}=\B_m\left[\oPi_nP_nv\right]\left(\oPi_nc_{nm}\right)\qquad\text{for } t\in(0,\infty)\backslash\left(\frac{1}{n}\N_0\right).
 \label{appr}
\end{align}
The definition of bounded weak solutions to \cref{ellpr} is as follows.
\begin{Definition}[Weak solution to \cref{ellpr}]\label{Defsolappr1e} Let $a$ and $\Phi$ be as defined in \cref{a} and \cref{Phi}, respectively, $f_c$ satisfy \cref{fc}, $m,n,T\in\N$, and 
\begin{subequations}\label{Assumpvc0}
\begin{align}
&0\leq c_0\in \LInf\cap \HOne ,\\
&0\leq v\in C([0,T];C^1_{\nu}(\overline{\Omega})).                                                                                                                                                                                      \end{align}
\end{subequations}
 We call a function $c_{nm}:\left(\frac{1}{n}\N_0\cap[0,T]\right)\times\overline{\Omega}\rightarrow\R_0^+$ a \underline{weak solution} to \cref{ellpr} if it satisfies 
\begin{align}
c_{nm}(t,\cdot)\in \LInf\cap\HOne\qquad\text{for all }t\in \frac{1}{n}\N_0\cap[0,T]\label{Regcnm}
\end{align}
and \cref{ellpr} in $(\HOne)^*$.

We call $c_{nm}:\frac{1}{n}\N_0\times\overline{\Omega}\rightarrow\R_0^+$ a \underline{global weak solution} to \cref{ellpr} if it is a weak solution for all $T\in\N$.
\end{Definition}

\begin{Lemma}[Weak solution to \cref{ellpr}]\label{LemSolAppr} Let $\chi\in(0,1)$. 
\begin{enumerate}
[label=(\arabic*),ref=(\arabic*),]
\item\label{itembndn}{\it (A priori bounds.)} Let $m,T\in\N$,
\begin{align*}
 \Cr{fc1}<n\in\N,
\end{align*}
and $(v,c_0)$ satisfy \cref{Assumpvc0}. Then interpolations $\oPi_nc_{nm}$ and $\Lambda_nc_{nm}$ of any corresponding weak solution $c_{nm}:\left(\frac{1}{n}\N_0\cap[0,T]\right)\times\overline{\Omega}\rightarrow\R_0^+$ to \cref{ellpr} satisfy the  estimate sets
\begin{subequations}\label{estoPicnm}
\begin{align}
 &\|\oPi_n{c_{nm}}\|_{L^{\infty}((0,T);\LInf)}
\leq\Cl{CBNDdiscr}\left(\frac{1}{n},T,\Cr{fc1},\|c_0\|_{\LInf}\right),\label{Picbounde}\\
&\|\oPi_n\nabla{c_{nm}}\|_{L^1((0,T);\LOne)}
 \leq\Cr{CD}\left(T,\Cr{CBNDdiscr}(\dots),\|v\|_{L^{\infty}((0,T);\LInf)}\right),\label{PiDCCe}\\
&\frac{1}{m}\|\oPi_n\nabla {c_{nm}}\|_{L^2((0,T);\LTwo)}^2\leq \Cr{CD}\left(T,\Cr{CBNDdiscr}(\dots),\|v\|_{L^{\infty}((0,T);\LInf)}\right),\label{DceL2}
\end{align} 
\end{subequations}
and
\begin{subequations}\label{estLacnm}
\begin{align} 
&\|\Lambda_n{c_{nm}}\|_{L^{\infty}((0,T);\LInf)}\leq \Cr{CBNDdiscr}(\dots),\label{Lacbounde}\\
&\|\Lambda_n\nabla{c_{nm}}\|_{L^1((0,T);\LOne)}\leq \Cr{CD}\left(T,\Cr{CBNDdiscr}(\dots),\|v\|_{L^{\infty}((0,T);\LInf)}\right) +\frac{1}{2n}\|\nabla c_0\|_{\LOne},\label{LaDCCe}\\
 &\|\partial_t\Lambda_n{c}_{nm}\|_{L^2((0,T);(\HOne)^*)}\leq \Cl{C22}\left(T,\Cr{CBNDdiscr}(\dots),\|v\|_{L^{\infty}((0,T);\LInf)}\right)\label{LaDcte},%
 \end{align} 
 \end{subequations}
 respectively,
 where
 \begin{align}
  \Cr{CBNDdiscr}(\dots)=&\Cr{CBND_1}\left(1+T\left(1-\frac{\Cr{fc1}}{n}\right)^{-1}\right)e^{T\Cr{fc1}}
 \|c_0\|_{\LInf}
  \underset{n\to\infty}{\to}\Cr{CBND}(T).\label{C20conv}
 \end{align}
\item\label{itemexn}{\it (Global existence.)} 
Let 
\begin{align}
 (\Cr{fc1})_++1<n\in\N\nonumber%
\end{align} 
and $(v,c_0)$ satisfy
\begin{align*}
&0\leq c_0\in \LInf\cap \HOne ,\\
&0\leq v\in C(\R_0^+;C^1_{\nu}(\overline{\Omega})).                                                                                                                                                                                      \end{align*}
Then for all $m\in\N$ system \cref{ellpr} possesses a global weak solution.
 \item\label{itemunn}{\it (Stability and uniqueness.)} Let $m,n,T\in\N$. 
 Set
 \begin{align}
 X(T,M):=&\left\{0\leq c_0\in \LInf\cap \HOne :\ \|c_0\|_{\LInf}\leq M\right\}\nonumber\\
 &\times\left\{0\leq v\in C([0,T];C^1_{\nu}(\overline{\Omega}))\cap L^1((0,T);W^{2,1}(\Omega)):\ \|v\|_{L^{\infty}((0,T);\LInf)}\leq M\right\}.\label{XTM}
\end{align}
  Let $c_{nm}^{(1)},c_{nm}^{(2)}:\left(\frac{1}{n}\N_0\cap[0,T]\right)\times\overline{\Omega}\rightarrow\R_0^+$ be some weak solutions to \cref{ellpr} corresponding to respective pairs  $\left(c_{0}^{(1)},v^{(1)}\right),\left(c_{0}^{(2)},v^{(2)}\right)\in X(T,M)$. 
  Set 
 \begin{align}
  M_n(T):=&\max\left\{M,\Cr{CBNDdiscr}\left(\frac{1}{n},T,\Cr{fc1},M\right)\right\}.
  \label{F4}
 \end{align}
 If 
\begin{align}
 n>\underset{[0,M_n(T)]^2}{\max}\partial_cf,\label{n3}
\end{align}
 then  
 \begin{align}
  &\Lambda_n\left\|c_{nm}^{(1)}-c_{nm}^{(2)}\right\|_{C([0,T];\LOne)}\nonumber\\
\leq&\exp_{n,\underset{[0,M_n(T)]^2}{\max}\partial_cf}(T)\left\|c_0^{(1)}-c_0^{(2)}\right\|_{\LOne}\nonumber\\
&+\Cl{C361}(M_n(T))\left(\left\|v^{(1)}-v^{(2)}\right\|_{C([0,T];C^1(\overline\Omega))}+\left\|v^{(1)}-v^{(2)}\right\|_{L^1((0,T);W^{2,1}(\Omega))}\right),\label{contice}
 \end{align}
 where $\exp_{n,\dots}$ is defined in \cref{LemDiscGr}. 
 In particular, if $v^{(1)}\equiv v^{(2)}$ and $c_{0}^{(1)}\equiv c_{0}^{(2)}$, then $c_{nm}^{(1)}\equiv c_{nm}^{(2)}$, so that in this case the solution is unique.

Furthermore, there exists some $N\in\N_0$ such that if $\left(c_0^{(1)},v^{(1)}\right),\left(c_0^{(2)},v^{(2)}\right)\in X(T,M)$ then \cref{n3} and, hence, \cref{contice} are valid for all $n>N$.
\end{enumerate}

\end{Lemma}
\begin{proof}
\begin{enumerate}[label=(\arabic*),ref=(\arabic*)]
 \item Let a weak solution $c_{nm}$ be given.    
Testing \cref{appr} with $g(\oPi_n{c_{nm}})$ for $g(u):=pu^{p-1}$, $p\geq2$ ($g(u):=\ln(\max\{u,{m}\})$) and arguing in a manner closely resembling the approach taken in the proof of \cref{Lemcbnd} (\cref{Lemnablac}), thereby using the tools provided in \cref{AppendixInter}, one obtains   \cref{Picbounde} and \cref{C20conv}  (\cref{PiDCCe}). Taking  $g(u):=u$ instead leads to  \cref{DceL2}. Bounds \cref{Lacbounde,LaDCCe}  
for $\Lambda_n c_{nm}$ are a direct consequence of \cref{Picbounde,PiDCCe}, respectively, combined with  \cref{LaPinorm}. 
Finally, using \cref{Picbounde,DceL2} and  properties of $a,\Phi,f$, one readily obtains   \cref{LaDcte}. 
Thus, any weak solution $c_{nm}$ of \cref{ellpr}  satisfies \cref{estoPicnm,estLacnm}. We omit further details.
\item Observe that
\begin{align*}
 f_c(T_{0,L}(c),v)\leq& \Cr{fc1}T_{0,L}(c)\\
 \leq& (\Cr{fc1})_+c\qquad\text{for all }L>0.
\end{align*}
Set
\begin{align}
 L_k:=\Cr{CBNDdiscr}\left(\frac{1}{n},\frac{k}{n},(\Cr{fc1})_+,\|c_0\|_{\LInf}\right)\qquad\text{for }k\in\N.\nonumber
\end{align}
From \cref{PropB} we know that $-\B_{n,L_k,m}$ 
is a  bounded coercive pseudo-monotone operator between $\HOne$ and its dual for all $k\in\N$. Applying a standard existence result for such operators (see, e.g. \cite[Chapter 2 \S2.4 Theorem 2.7]{Lions}), we recursively deduce the existence of weak solutions 
\begin{align*}
 c_{nm}\left(\frac{k}{n},\cdot\right)\in\HOne\qquad\text{ for } k\in\N,
\end{align*}
to the sequence of elliptic problems 
\begin{align}
 &\frac{c_{nm}\left(\frac{k}{n},\cdot\right)-c_{nm}\left(\frac{k-1}{n},\cdot\right)}{\frac{1}{n}}=\B_{L_k,m}\left[P_nv\left(\frac{k}{n},\cdot\right)\right]\left(c_{nm}\left(\frac{k}{n},\cdot\right)\right)\qquad\text{for }k\in\N\label{ellprL}
\end{align}
under zero Neumann boundary conditions. 
 Well-known results on boundedness (see, e.g. \cite{LU}) 
for such elliptic equations  
ensure \cref{Regcnm} and the non-negativity of $c_{nm}$.  

For each $k\in\N$ we can apply the result of \cref{itembndn} for $T:=k/n$ and $f_c$ replaced by $f_c(T_{0,L_k}(\cdot),\cdot)$, yielding 
\begin{align*}
 c_{nm}\left(\frac{k}{n},\cdot\right)\leq L_k\qquad\text{for all }k\in\N.
\end{align*}
 Therefore \cref{appr1e} and \cref{ellprL} are equivalent, and so $c_{nm}$ is a global weak solution to \cref{ellpr}.
 \item
 Let $c_{nm}^{(1)},c_{nm}^{(2)}:\left(\frac{1}{n}\N_0\cap[0,T]\right)\times\overline{\Omega}\rightarrow\R_0^+$ be some weak solutions to \cref{ellpr} corresponding to respective pairs  $\left(c_{0}^{(1)},v^{(1)}\right),\left(c_{0}^{(2)},v^{(2)}\right)\in X(T,M)$ and assume that \cref{n3} holds. 
 Subtracting \cref{appr1e} for  $c_{nm}^{(1)}$ and $c_{nm}^{(2)}$, we  test the resulting equation with $$g(u):=\sign_{\delta}(u)$$ for $\delta>0$ and $$u:=c_{nm}^{(1)}-c_{nm}^{(2)}.$$
 Since both solutions satisfy \cref{Regcnm} and $\sign_{\delta}$ is Lipschitz, $g(u)$ is a valid test function. Making use of \cref{propg}, we estimate as follows:
 \begin{align}
  &\frac{d}{dt}\Lambda_n\int_{\Omega}\int_0^{c_{nm}^{(1)}-c_{nm}^{(2)}}\sign_{\delta}(s)\,ds\,dx\nonumber\\
  \leq &\oPi_n\left<\A[P_n v^{(1)}]\left(c_{nm}^{(1)}\right)-\A[P_n v^{(2)}]\left(c_{nm}^{(2)}\right), \sign_{\delta}\left(c_{nm}^{(1)}-c_{nm}^{(2)}\right)\right>\nonumber\\
  &-\oPi_n\frac{1}{m}\int_{\Omega}\sign_{\delta}'\left(c_{nm}^{(1)}-c_{nm}^{(2)}\right)\left|\nabla\left(c_{nm}^{(1)}-c_{nm}^{(2)}\right)\right|^2\,dx\nonumber\\
  &+\oPi_n\int_{\Omega}\left(f_c\left(c_{nm}^{(1)},P_nv^{(1)}\right)-f_c\left(c_{nm}^{(2)},P_nv^{(2)}\right)\right)\sign_{\delta}\left(c_{nm}^{(1)}-c_{nm}^{(2)}\right)\,dx\nonumber\\
  \leq &\oPi_n\left<\A[P_n v^{(1)}]\left(c_{nm}^{(1)}\right)-\A[P_n v^{(2)}]\left(c_{nm}^{(2)}\right), \sign_{\delta}\left(c_{nm}^{(1)}-c_{nm}^{(2)}\right)\right>\nonumber\\
  &+\oPi_n\int_{\Omega}\left(f_c\left(c_{nm}^{(1)},P_nv^{(1)}\right)-f_c\left(c_{nm}^{(2)},P_nv^{(2)}\right)\right)\sign_{\delta}\left(c_{nm}^{(1)}-c_{nm}^{(2)}\right)\,dx.\label{est4_1}
  \end{align}
  Notice that the left-hand side of \cref{est4_1} is time-independent in each time interval $((k-1)/n,k/n)$, $k\in\N_0$. 
  This allows to pass to the limit superior as $\delta\to0$ on both sides of \cref{est4_1} using \cref{Kato2} and, where necessary,  the dominated convergence theorem together with \cref{Picbounde} and properties of $\A$, $f$, $\sign_{\delta}$, and $v$. Thus obtained inequality and estimates based on 
  \cref{Picbounde}, \cref{F4}, and \cref{PiP_} yield
 \begin{align}
  &\frac{d}{dt}\Lambda_n\left\|c_{nm}^{(1)}-c_{nm}^{(2)}\right\|_{\LOne}\nonumber\\\leq&\Cr{C19}\oPi_n\left(\left\|\nabla c_{nm}^{(2)}\right\|_{\LOne}\left\|P_n(v^{(1)}-v^{(2)})\right\|_{C^1(\overline\Omega)}+\left\|c_{nm}^{(2)}\right\|_{\LInf}\left\|P_n(v^{(1)}-v^{(2)})\right\|_{W^{2,1}(\Omega)}\right)\nonumber\\
  &+\oPi_n\int_{\Omega}\left(f_c\left(c_{nm}^{(1)},P_nv^{(1)}\right)-f_c\left(c_{nm}^{(2)},P_nv^{(2)}\right)\right)\sign\left(c_{nm}^{(1)}-c_{nm}^{(2)}\right)\,dx\nonumber\\
  \leq& \Cr{C19}\oPi_n\left(\left\|\nabla c_{nm}^{(2)}\right\|_{\LOne}\left\|P_n(v^{(1)}-v^{(2)})\right\|_{C^1(\overline\Omega)}+\left\|c_{nm}^{(2)}\right\|_{\LInf}\left\|P_n(v^{(1)}-v^{(2)})\right\|_{W^{2,1}(\Omega)}\right)\nonumber\\
  &+\underset{[0,M_n(T)]^2}{\max}|\partial_vf|\oPi_n\left\|P_n(v^{(1)}-v^{(2)})\right\|_{\LOne}+\underset{[0,M_n(T)]^2}{\max}\partial_cf\oPi_n\left\|c_{nm}^{(1)}-c_{nm}^{(2)}\right\|_{\LOne}\nonumber\\
  \leq&\Cr{C19}\oPi_n\left\|\nabla c_{nm}^{(2)}\right\|_{\LOne}\left\|v^{(1)}-v^{(2)}\right\|_{C([0,T];C^1(\overline\Omega))}+\Cr{C19}(M_n)\oPi_n\left\|P_n(v^{(1)}-v^{(2)})\right\|_{W^{2,1}(\Omega)}\nonumber\\
  &+\underset{[0,M_n(T)]^2}{\max}\partial_cf\oPi_n\left\|c_{nm}^{(1)}-c_{nm}^{(2)}\right\|_{\LOne}\qquad\text{in } (0,T)\backslash\left(\frac{1}{n}\N_0\right).\label{est13}
 \end{align}
 Due to \cref{n3} \cref{LemDiscGr} applies to \cref{est13}. Together with    \cref{PiDCCe}, \cref{oPinorm}, \cref{PiP_}, and \cref{LaLInf} it leads to \cref{contice}. 

Choosing 
\begin{align*}
 \left(c_{0}^{(1)},v^{(1)}\right):=\left(c_{0}^{(2)},v^{(2)}\right):=(c_0,v)
\end{align*} 
in \cref{contice}  yields $c_{mn}^{(1)}\equiv c_{mn}^{(2)}$, i.e. the solution is unique. 

Finally, \cref{C20conv} and $f\in C^1$ imply that 
\begin{align}
 \left(\underset{\left[0,\max\left\{M,\Cr{CBNDdiscr}\left(\frac{1}{n},T,\Cr{fc1},M\right)\right\}\right]^2}{\max}\partial_cf\right)\nonumber
\end{align}
 is convergent. Therefore, for sufficiently large $N=N(T,M)$ condition  \cref{n3} is met for all $n>N$ and $\left(c_0^{(1)},v^{(1)}\right),\left(c_0^{(2)},v^{(2)}\right)\in X(T,M)$. 
\end{enumerate}
\end{proof}
\subsection{Weak  solutions to IBVP  \texorpdfstring{\cref{IBVPc}}{}} \label{Secsolc}
Now we are ready to construct  entropy solutions to the IBVP \cref{IBVPc} for fixed $v$. We formulate a result in terms of a solution mapping as this is needed in \cref{SecSchauder}. 

\begin{Theorem}[Entropy solutions to \cref{IBVPc}]\label{thmSolEq1} 
Let $\chi\in(0,1)$ and assumptions \cref{a,Phi,fc} hold. Then there exists a mapping $\Ss_1$ that
\begin{enumerate}
[label=(\arabic*),ref=(\arabic*)]
\item assigns any pair $(c_0,v)$ that satisfies \cref{c0bnd,classvinc} to some $c$ such that $(c,v)$ solves \cref{IBVPc} in terms of \cref{Defsol1} and satisfies estimates \cref{cbound} and, for all $T\in\N$,
\begin{align}
 &\|D c\|_{L^1_{w}((0,T);\RM)}\leq \Cr{CD}\left(T,\Cr{CBND}(T),\|v\|_{L^{\infty}((0,T);\LInf)}\right);\label{DCC_}
\end{align}
\item is locally Lipschitz continuous in the following sense: if $c^{(1)}=\Ss_1\left(c^{(1)}_0,v^{(1)}\right)$ and $c^{(2)}=\Ss_1\left(c^{(2)}_0,v^{(2)}\right)$ for some
\begin{subequations}\label{c0vreg}
\begin{align}
 &0\leq c^{(1)}_0,c^{(2)}_0\in\LInf,\\
 &0\leq v^{(1)},v^{(2)}\in C(\R_0^+;C^1_{\nu}(\overline{\Omega}))\cap L^1_{loc}(\R_0^+;W^{2,1}(\Omega)),
\end{align}
\end{subequations}
then for all $T\in\N$ and $t\in[0,T]$
\begin{subequations}\label{contic}
 \begin{align}
 &\left\|c^{(1)}-c^{(2)}\right\|_{\LOne}\nonumber\\
 \leq&e^{T\underset{(0,M(T))^2}{\max}\partial_cf}\left\|c_0^{(1)}-c_0^{(2)}\right\|_{\LOne}\nonumber\\
&+\Cr{C361}(M(T))\left(\left\|v^{(1)}-v^{(2)}\right\|_{C([0,T];C^1(\overline\Omega))}+\left\|v^{(1)}-v^{(2)}\right\|_{L^1((0,T);W^{2,1}(\Omega))}\right),
 \end{align}
 where 
\begin{align}
  M(T):=&\max\left\{\left\|(v^{(1)},v^{(2)})\right\|_{(L^{\infty}((0,T);\LInf))^2},\Cr{CBND}(T)\right\}.\label{MT}
 \end{align}
\end{subequations}                  \end{enumerate}

\end{Theorem}
\begin{Remark}
 \cref{thmSolEq1}  provides existence of a solution class that depends on the data in a stable fashion. This is sufficient for our needs as it allows to prove the main \cref{mainthm}. We do not address uniqueness of solutions to \cref{IBVPc} in this work. 
\end{Remark}

\begin{proof}[{\it {\bf Proof of} \cref{thmSolEq1}}]
To shorten the notation, we introduce 
\begin{align}
 J_m[c,v]:=\frac{1}{m}\nabla c+a(c,\nabla c)-\chi c \Phi(\nabla v).\label{defJm}
\end{align}
Throughout the proof we assume
\begin{align}
&T\in\N,%
\nonumber\\
  &0\leq\rho\in D(\Omega),\qquad 0\leq \eta\in D((0,T)),\label{rhoeta}\\
 &g:\R_0^+\rightarrow\R\text{ increasing and locally Lipschtz},\qquad G(c):=\int_0^cg(u)\,du. \label{asstestG}                                                                                                                                                                         \end{align}
 \begin{proofpart}[Equation for the evolution of $c$]\label{PP1} Let $(c_0,v)$  be such that  \cref{c0bnd,classvinc} holds.
 Then $c_0$ can be approximated by   a sequence $(c_{n0})$  that satisfies
 \begin{subequations}
\begin{align}
 &0\leq c_{n0}\in C^1(\overline{\Omega}),\label{regcn0}\\
 &\|c_{n0}\|_{\LInf}\leq \|c_0\|_{\LInf},\label{cn0bnd}\\ 
 &\|c_{n0}\|_{\WOne}\leq \sqrt{n}\Cl{nabcn0}\|c_0\|_{\LInf}, \label{nabc0}\\
 &c_{n0}\underset{n\rightarrow\infty}{\overset{*}{\rightharpoonup}}  c_0\qquad \text{in }\LInf.%
 \label{iniconv}
\end{align} 
\end{subequations}
Let $n>(\Cr{fc1})_++1$ and $m\in\N$.
Due to \cref{LemSolAppr} \cref{itembndn,itemexn} and assumptions {\cref{regcn0,classvinc,nabc0,cn0bnd}},  the time-discrete approximation \cref{ellpr} possesses a global weak solution (in terms of \cref{Defsolappr1e}) $c_{nm}$ that                                                                                                                                                                                                                                                                                                                                                                                                                  corresponds to  $(c_{n0},v)$ and satisfies for all $T\in\N$  estimates \cref{estoPicnm,LaDcte,Lacbounde} corresponding to $(c_0,v)$ as well as 
\begin{align}
 \|\Lambda_n\nabla{c_{nm}}\|_{L^1((0,T);\LOne)}\leq &\Cr{CD}\left(T,\Cr{CBNDdiscr}(\dots),\|v\|_{L^{\infty}((0,T);\LInf)}\right) +\frac{\Cr{nabcn0}}{2\sqrt{n}}\|c_0\|_{\LInf}\nonumber\\
 \limn&\Cr{CD}\left(T,\Cr{CBND}(T),\|v\|_{L^{\infty}((0,T);\LInf)}\right).\label{LaDCCen}
\end{align}
In light of these facts, there exist 
 functions $(c,z):\R_0^+\times \overline{\Omega}\rightarrow\R_0^+\times (\R_0^+)^d$, with $c$ satisfying \cref{regc},
 and two sequences $(n_l)$ and $(m_l)$, such that:\\
 due to \cref{Lacbounde}, \cref{LaDcte}, and  \cref{LaDCCen}, a version of the Lions-Aubin lemma \cite[Section 8 Corollary 4]{Simon}, and the diagonal argument (which we use in what follows without referring to it explicitly)
\begin{alignat}{3}
  &\Lambda_{n_l} c_{n_lm_l}\underset{l\rightarrow\infty}{\rightarrow} c&&\qquad \text{in }L^1_{loc}(\R_0^+;\LOne),\label{LaconveL1}\\
  &&&\qquad \text{in }\LOne\text{ a.e. in }(0,\infty),\label{convL1ae}\\
  &\cref{cbound}\text{ holds};&&\nonumber
\end{alignat} 
due to \cref{Lacbounde} and \cref{LaconveL1} and H\"older's inequality
\begin{align}
  &\Lambda_{n_l} c_{n_lm_l}\underset{l\rightarrow\infty}{\rightarrow} c\qquad \text{in }L^2_{loc}(\R_0^+;\LTwo);
  \label{LaconveL2}
\end{align} 
due to \cref{LaconveL2,equivLa2}
\begin{align}
  \oPi_{n_l} c_{n_lm_l}\underset{l\rightarrow\infty}{\rightarrow} c&\qquad \text{in }L^2_{loc}(\R_0^+;\LTwo),\nonumber
  \\
  &\qquad \text{a.e. in }(0,\infty)\times \Omega;
  \label{Piconvae}
\end{align} 
due to \cref{Picbounde}, \cref{Piconvae}, and \cref{convPn}, continuity of $\Phi$, $v$, and $\nabla v$, and the dominated convergence theorem 
\begin{alignat}{3}
  &\oPi_{n_l} F(c_{n_lm_l},P_{n_l}v)\underset{l\rightarrow\infty}{\rightarrow} F(c,v)&&\qquad \text{in }L^2_{loc}(\R_0^+;\LTwo)\qquad\text{for all }F\in C(\R\times\overline{\Omega})\label{PiconvvL1},\\
  &\oPi_{n_l} F(c_{n_lm_l},\Phi(P_{n_l}\nabla v))\underset{l\rightarrow\infty}{\rightarrow} F(c,\Phi(\nabla v))&&\qquad \text{in }L^2_{loc}(\R_0^+;\LTwo)\qquad\text{for all }F\in C(\R\times\overline{\Omega})\label{PiconvnvL1};
\end{alignat} 
due to \cref{PiconvvL1}
\begin{alignat}{3}
  &\oPi_{n_l} g(c_{n_lm_l})\underset{l\rightarrow\infty}{\rightarrow} g(c)&&\qquad \text{in }L^2_{loc}(\R_0^+;\LTwo)\label{PiconvgvL1},\\
  &&&\qquad \text{in }\LOne\text{ a.e. in }(0,\infty);\label{PiconvgvL1ae}
\end{alignat}
due to \cref{Picbounde}, \cref{PiconvvL1}, \cref{PiconvnvL1}, and \cref{PiP_}, boundedness of $v$, and the Banach-Alaoglu theorem
\begin{alignat}{3}
  &\oPi_{n_l} F(c_{n_lm_l},P_{n_l}v)\underset{l\rightarrow\infty}{\overset{*}{\rightharpoonup}}  F(c,v)&&\qquad \text{in }L^{\infty}_{loc}(\R_0^+;\LInf)\qquad\text{for all }F\in C(\R\times\overline{\Omega})\label{PiconvvwLi},\\
  &\oPi_{n_l} F(c_{n_lm_l},\Phi(P_{n_l}\nabla v))\underset{l\rightarrow\infty}{\overset{*}{\rightharpoonup}}  F(c,\Phi(\nabla v))&&\qquad \text{in }L^{\infty}_{loc}(\R_0^+;\LInf)\qquad\text{for all }F\in C(\R\times\overline{\Omega})\label{PiconvnvwLi};
\end{alignat} 
due to  \cref{cn0bnd}, \cref{PiconvgvL1}, and \cref{equivLa2}
\begin{align}
  \Lambda_{n_l} g(c_{n_lm_l})\underset{l\rightarrow\infty}{\rightarrow} g(c)&\qquad \text{in }L^2_{loc}(\R_0^+;\LTwo);\label{LaGconveL2}
\end{align}
due to \cref{RMKV}, \cref{Picbounde}, \cref{PiDCCe}, and \cref{PiconvgvL1}, the chain rule, and the Banach-Alaoglu theorem
\begin{align}
&\nabla\oPi_{n_l}g\left(c_{n_lm_l}\right)\underset{l\rightarrow\infty}{\overset{*}{\rightharpoonup}}D g(c)\qquad\text{in }{\cal M}_{loc}(\R_0^+;\RM);\label{convDc2}
\end{align}
due to \cref{Picbounde}, \cref{PiDCCe}, and \cref{PiconvgvL1ae}, the chain rule, and the lower semicontinuity of the total variation with respect to the convergence in $\LOne$
\begin{align}
 &g(c)(t,\cdot)\in \BV,\label{c2tBV}\\ 
 &\left\|Dg(c)(t,\cdot)\right\|_{\RM}\leq \underset{l\rightarrow\infty}{\lim\inf}\left\|\nabla\oPi_{n_l}g\left( c_{n_lm_l}\right)(t,\cdot)\right\|_{\LOne}<\infty\ \ \ \text{a.e. in }(0,\infty);\nonumber
\end{align}
due to $g(c)\in L^1_{loc}(\R_0^+;\LOne)$, \cref{c2tBV},  and \cite[Lemma 5]{ACM02}
\begin{align}
 g(c):(0,\infty)\to \BV\qquad\text{weakly measurable};\label{wm}
\end{align}
due to $g(c)\in L^1_{loc}(\R_0^+;\LOne)$, $D g(c)\in{\cal M}_{loc}(\R_0^+;\RM)$, and \cref{wm}  
\begin{align}
  &g(c)\in L^1_{w,loc}(\R_0^+;\BV),\label{Dc2}\\
  &\cref{DCC_}\text{ holds};\nonumber
\end{align}
due to \cref{Dc2,LemMofParts}
\begin{subequations}
\begin{align}
  &\nabla g(c)\in L^1_{loc}(\R_0^+;\LOne),\label{nDc2str}\\
  &D^s g(c)\in L^1_{w,loc}(\R_0^+;\RM),\label{nDc3}\\
  &|D^s g(c)|\in L^1_{w-*,loc}(\R_0^+;\RM);\label{nDc4}
\end{align}
\end{subequations}
due to \cref{cbound,c2tBV} and the chain rule for the superposition of a Lipschitz and a BV functions
\begin{subequations}\label{chaingc}
\begin{alignat}{3}
 &\nabla g(c)(t,\cdot)=g'(c)\nabla c(t,\cdot)&&\qquad\text{a.e. in }\Omega \text{ for a.a. }t\in(0,\infty),\\
 &|D^sg(c)(t,\cdot)|\leq \|g'\|_{L^{\infty}((0,\Cr{CBND}(T)))}|D^sc(t,\cdot)|&&\qquad\text{in }\RM\text{ for a.a. }t\in(0,\infty);
\end{alignat}
\end{subequations}
due to $\partial_t c\in L^{\infty}_{loc}(\R_0^+;(\WOne)^*)$ and $c_0\in(\WOne)^*$
\begin{align}
 c\in C(\R_0^+;(\WOne)^*)\label{ccontWm1};
\end{align}
due to  \cref{ccontWm1}, $c\in L^{\infty}_{loc}(\R_0^+;\LInf)$, and a result on weak continuity  \cite[Chapter 3 \S1.4 Lemma~1.4]{Temam}
\begin{align}
c\in C_w(\R_0^+;\LInf);\nonumber%
\end{align}
due to \cref{abound,Picbounde} and the Banach-Alaoglu theorem 
\begin{alignat}{3}
 &\oPi_{n_l} a\left(c_{n_lm_l},\nabla c_{n_lm_l}\right)\underset{l\rightarrow\infty}{\overset{*}{\rightharpoonup}} z&&\qquad\text{in }(L^{\infty}_{loc}(\R_0^+;\LInf))^d,\label{conva}\\
 &\|z\|_{(L^{\infty}((0,T);\LInf))^d}\leq\Cr{CBND}(T)&&\qquad\text{for all }T>0;\label{zbound}
\end{alignat}
due to                            \cref{DceL2}, \cref{defJm}, \cref{PiconvnvwLi}, and  \cref{conva}
\begin{align}
 \oPi_{n_l}J_{m_l}[c_{m_ln_l},P_{n_l}v]\underset{l\rightarrow\infty}{\overset{*}{\rightharpoonup}} z-\chi c\Phi(\nabla v)
 =:&J\qquad\text{in }(L^{\infty}_{loc}(\R_0^+;\LInf))^d;\label{convJ}
\end{align}
due to \cref{Phi1}, \cref{cbound}, and \cref{zbound}
\begin{align}
 \|J\|_{L^{\infty}((0,T);\LInf)}\leq (1+\chi)\Cr{CBND}(T)%
 ;\label{Jnorm}
\end{align}
due to  \cref{PiconvvwLi} and continuity of $f_c$
\begin{alignat}{3}
  &\oPi_{n_l} f_c(c_{n_lm_l},P_{n_l}v)g(c_{n_lm_l})\underset{l\rightarrow\infty}{\overset{*}{\rightharpoonup}}  f_c(c,v)g(c)&&\qquad \text{in }L^{\infty}_{loc}(\R_0^+;\LInf);\label{PifconveLinf}
\end{alignat} 
due to \cref{Bm}, \cref{appr,convJ,PifconveLinf,LaconveL2}
\begin{align}
  &\partial_t \Lambda_{n_l} c_{n_l}\underset{l\rightarrow\infty}{\overset{*}{\rightharpoonup}}\partial_t c\qquad\text{in }L^{\infty}_{loc}(\R_0^+;(\WOne)^*)%
  \label{convct}
\end{align}
and for all $\varphi\in  L^1_{loc}(\R_0^+;\WOne)$
\begin{align}
 &\left<\partial_tc,\varphi\right>=\int_{\Omega}-J\cdot\nabla\varphi+f_c(c,v)\varphi\,dx\qquad\text{a.e. in }(0,\infty);\label{limeq}
\end{align}
due to \cref{limeq}
\begin{align}
 \partial_tc=\nabla\cdot J+f_c(c,v)\qquad\text{in }W^{-1,\infty}(\Omega)\qquad\text{ a.e. in }(0,\infty);\label{limeq_}
\end{align}
due to \cref{iniconv}, \cref{LaconveL2}, and \cref{convct} and the trace theorem for Bochner integrals
\begin{align*}
  &\Lambda_{n_l} c_{n_lm_l}(0,\cdot)\underset{l\rightarrow\infty}{\overset{*}{\rightharpoonup}}   c(0,\cdot)=c_0\qquad \text{in }\LInf.
\end{align*}

\end{proofpart}
\medskip
We also need certain inequalities that control  evolution of  $G(c)$ for $G$ as in \cref{asstestG}. Two sets of such inequalities are  derived in the next Steps.
\begin{proofpart}[Inequalities for the evolution of $G(c)$: part I]\label{ppI}
In this part of the proof we verify that 
\begin{subequations}\label{mugen}
\begin{align}
  \partial_t G(c)+J\cdot \nabla g(c)-\nabla\cdot(g(c)J)-f_c(c,v)g(c)=:\mu_g\label{evolgc}
 \end{align}
with
\begin{align}
 &\mu_g\in L^1_{w-*,loc}(\R_0^+;\RM),\label{mug}\\
 &|\mu_g|\leq (1+\chi)\Cr{CBND}(T)|D^sg(c)|\qquad\text{in }\RM\text{ a.e. in }(0,T)
 .\label{est33_}
\end{align}
\end{subequations}

To begin with, we utilise \cref{limeq_}.  To avoid dealing with the distributional time derivative of $c$ directly we again avail  of Steklov averages. 
 Let $t>0$ and $\tau\in(0,1)$.  
Integrating both sides of equation \cref{limeq_} over  $[t,t+\tau]$ in Bochner's sense, dividing by $\tau$,  and using $c\in C_w(\R_0^+;\LInf)\cap L^{\infty}_{loc}(\R_0^+;\LInf)$, $J\in L^{\infty}_{loc}(\R_0^+;(\LInf)^d)$, and the continuity of $v$ and $f_c$, we obtain that
\begin{align}
 \frac{1}{\tau}(c(t+\tau,\cdot)-c(t,\cdot))=\nabla\cdot\left(\frac{1}{\tau}\int_t^{t+\tau}J\,ds\right)+\frac{1}{\tau}\int_t^{t+\tau}f_c(c,v)\,ds\qquad \text{in }\LInf
 \text{ for all }t\in\R_0^+
 .\label{term2}
\end{align}
Observe that $\frac{1}{\tau}\int_{\cdot}^{\cdot+\tau}J\,ds\in L^{\infty}_{loc}(\R_0^+;(\LInf)^d)$ due to $J\in L^{\infty}_{loc}(\R_0^+;(\LInf)^d)$, whereas, 
in consequence of \cref{term2}, $\nabla\cdot\left(\frac{1}{\tau}\int_{\cdot}^{\cdot+\tau}J\,ds\right) \in L^{\infty}_{loc}(\R_0^+;\LInf)$. Hence, 
\begin{align}
\frac{1}{\tau}\int_{\cdot}^{\cdot+\tau}J\,ds\in L^{\infty}_{loc}(\R_0^+;X_{\infty}(\Omega)),\label{L1wXinf}
\end{align}
 where $X_{\infty}(\Omega)$ is the space defined in \cref{Xinf}. Let us now consider the time-dependent measures 
 \begin{align}
  M_{\theta\tau}(t,\cdot):=\left(\frac{1}{\tau}\int_{t-\theta\tau}^{t+(1-\theta)\tau}J\,ds,Dg(c)(t,\cdot)\right)\qquad\text{for } \theta\in\{0,1\}.\nonumber
 \end{align}
 In the following bounds and convergences are understood to hold for both choices of $\theta$. 
 Set 
 \begin{align}
  &J:=0\qquad\text{in }(-\infty,0)\times\Omega. \label{J0}
 \end{align}
 Combining \cref{L1wXinf,Dc2} and the fact that  pairing $(\cdot,\cdot)$ is a bounded bilinear form (see \cref{pairingdef} in \cref{SecBV}) produces 
\begin{subequations}
\begin{align}
 M_{\theta\tau}\in L^1_{w,loc}(\R_0^+;{\RM}).%
\label{integr1}  
\end{align}
Utilising \cref{ZDwa} we compute   
\begin{align}
 M_{\theta\tau}^{ac}(t,\cdot)=&\frac{1}{\tau}\int_{t-\theta\tau}^{t+(1-\theta)\tau}J\,ds\cdot \nabla g(c)(t,\cdot)\qquad\text{a.e. in }\Omega\text{ for a.a. }t\in(0,\infty).\label{acp11}
\end{align}
Further, due to  \cref{ZDws,Jnorm,J0}
\begin{align}
 \left|M_{\theta\tau}^{s}(t,\cdot)\right|\leq &\left\|\frac{1}{\tau}\int_{t-\theta\tau}^{t+(1-\theta)\tau}J\,ds\right\|_{\LInf}|D^s g(c)(t,\cdot)|\nonumber\\
 \leq &(1+\chi)\Cr{CBND}(T)|D^s g(c)(t,\cdot)|\qquad\text{in }\RM\text{ for a.a. }t\in(0,T)
 .\label{sp1}
\end{align}
\end{subequations}
Due to \cref{sp1,RMKV} and the Banach-Alaoglu theorem there exist a $\theta$-independent sequence  
\begin{align*}
(\tau_l)\subset(0,1),  \qquad\tau_l\rightarrow0,                                                                                                                           \end{align*}
 and some $M_{\theta}^{s}$, such that
\begin{subequations}
\begin{alignat}{3}
 &M_{\theta\tau_l}^{s}\underset{l\rightarrow\infty}{\overset{*}{\rightharpoonup}}M_{\theta}^{s}&&\qquad\text{in }{\cal M}_{loc}(\R_0^+;\RM),\label{spconv}
 \\
 &|M_{\theta}^{s}(E,\cdot)|\leq (1+\chi)\Cr{CBND}(T)|D^sg(c)(E,\cdot)|&&\qquad\text{in }\RM\qquad\text{for all Borel }E\subset [0,T].\label{Msbnd}
\end{alignat}
\end{subequations}
In particular, \cref{Msbnd,nDc3} imply that the assumptions of \cref{LemRep}\cref{LemRep2} are satisfied for $\mu:=M_{\theta}^{s}$ and $g:=(1+\chi)\Cr{CBND}(T)\|D^sc\|_{\RM}$. In view of this we conclude that
\begin{align}
 M_{\theta}^{s}\in L^1_{w-*,loc}(\R_0^+;\RM),\label{Mth}
\end{align}
and \cref{Msbnd} can then be replaced by
\begin{align}
 |M_{\theta}^{s}(t,\cdot)|\leq (1+\chi)\Cr{CBND}(T)|D^sg(c)|(t,\cdot)\qquad\text{in }\RM \text{ for a.a. }t\in(0,T)
 . \label{Msbnd1}
\end{align}
Let 
\begin{align}
 &0\leq\varphi\in C^1([0,T]\times\overline{\Omega})\cap C_0([0,T],\Co)
 .\label{assvar}
\end{align}
Combining \cref{sp1,nDc4}, we deduce that for $\varphi$ as in \cref{assvar}
\begin{align}
\left|\int_0^{\tau_l}\left<M_{1\tau_l}^{s}(t,\cdot),\varphi(t,\cdot)\right>dt\right|\leq &(1+\chi)\Cr{CBND}(T)\|\varphi\|_{L^{\infty}((0,{T});\LInf)}\int_0^{\tau_l}\|D^sc(t,\cdot)\|_{\RM}\,dt\nonumber\\
 \underset{l\to\infty}{\to}&0,\nonumber\\
\left|\int_{T-\tau_l}^{T}\left<M_{0\tau_l}^{s}(t,\cdot),\varphi(t,\cdot)\right>dt\right|\leq&(1+\chi)\Cr{CBND}(T)\|\varphi\|_{L^{\infty}((0,{T});\LInf)}\int_{T-\tau_l}^{T}\|D^sc(t,\cdot)\|_{\RM}\,dt\nonumber\\
 \underset{l\to\infty}{\to}&0,\nonumber
\end{align}
which together with \cref{spconv}
yields
\begin{align}
 \int_{\theta\tau_l}^{{T}-(1-\theta)\tau_l}\left<M_{\theta\tau_l}^{s}(t,\cdot),\varphi(t,\cdot)\right>dt
 \underset{l\to\infty}{\to}&\int_0^{T}\left<M^{s}_{\theta}(t,\cdot),\varphi(t,\cdot)\right>dt.\label{conv3}
\end{align}

Now we can proceed with \cref{term2}. Multiplying this equality by $(-\varphi g(c))(t+\theta\tau,\cdot)$ with $\varphi$ as in \cref{assvar}, applying \cref{prod}, integrating  with respect to $t$ over $(0,{T}-\tau)$ (recall that ${T}\geq1>\tau$), and using \cref{integr1,acp11}, we conclude that   
\begin{align}
&-\int_0^{{T}-\tau}\int_{\Omega}\varphi(t+\theta\tau,\cdot)\frac{1}{\tau}(c(t+\tau,\cdot)-c(t,\cdot))g(c)(t+\theta\tau,\cdot)\,dxdt\nonumber\\
=&\int_0^{{T}-\tau}\frac{1}{\tau}\int_t^{t+\tau}\int_{\Omega}\nabla\varphi(t+\theta\tau,\cdot)\cdot J(s,\cdot)g(c)(t+\theta\tau,\cdot)-\varphi(t+\theta\tau,\cdot)f_c(c,v)(s,\cdot)\, g(c)(t+\theta\tau,\cdot) \,dxdsdt\nonumber\\
&+\int_0^{{T}-\tau}\left<\left(\frac{1}{\tau}\int_{t}^{t+\tau}J\,ds,Dg(c)(t+\theta\tau,\cdot)\right),\varphi(t+\theta\tau,\cdot)\right>dt\nonumber\\
=&\int_{\theta\tau}^{{T}-(1-\theta)\tau}\frac{1}{\tau}\int_{t-\theta\tau}^{t+(1-\theta)\tau}\int_{\Omega}\left(\varphi(t,\cdot) J(s,\cdot)\cdot \nabla g(c)(t,\cdot)+\nabla\varphi(t,\cdot) \cdot J(s,\cdot) g(c)(t,\cdot)\right.\nonumber\\
&\phantom{\int_{\theta\tau}^{{T}-(1-\theta)\tau}\frac{1}{\tau}\int_{t-\theta\tau}^{t+(1-\theta)\tau}\int_{\Omega}(}\left.-\varphi(t,\cdot)f_c(c,v)(s,\cdot)\, g(c)(t,\cdot)\right) \,dxdsdt\nonumber\\
&+\int_{\theta\tau}^{{T}-(1-\theta)\tau}\left<M_{\theta\tau}^{s}(t,\cdot),\varphi(t,\cdot)\right>dt.\label{cttau_}
\end{align}
Since $c\in L^{\infty}_{loc}(\R_0^+;\LInf)$, $J\in L^{\infty}_{loc}(\R_0^+;(\LInf)^d)$, $\nabla c\in L^1_{loc}(\R_0^+;\LOne)$, and $v,f_c,\varphi,\nabla\varphi$ are continuous, the Lebesgue differentiation and the  dominated convergence  theorems can be used on the right-hand side of
\cref{cttau_}  along with \cref{conv3} yielding 
\begin{align}
&\int_{\theta\tau_l}^{{T}-(1-\theta)\tau_l}\frac{1}{\tau_l}\int_{t-\theta\tau_l}^{t+(1-\theta)\tau_l}\int_{\Omega}\left(\varphi(t,\cdot) J(s,\cdot)\cdot \nabla g(c)(t,\cdot)+\nabla\varphi(t,\cdot) \cdot J(s,\cdot) g(c)(t,\cdot)\right.\nonumber\\
&\phantom{\int_{\theta\tau_l}^{{T}-(1-\theta)\tau_l}\frac{1}{\tau_l}\int_{t-\theta\tau_l}^{t+(1-\theta)\tau_l}\int_{\Omega}(}\left.-\varphi(t,\cdot)f_c(c,v)(s,\cdot)\, g(c)(t,\cdot)\right) \,dxdsdt\nonumber\\
&+\int_{\theta\tau_l}^{{T}-(1-\theta)\tau_l}\left<M_{\theta\tau_l}^{s}(t,\cdot),\varphi(t,\cdot)\right>dt\nonumber\\
\underset{l\to\infty}{\to}&\int_0^{T}\int_{\Omega}\varphi J\cdot \nabla g(c)+\nabla\varphi\cdot J g(c)-\varphi f_c(c,v)\, g(c)\,dxdt+\int_0^{T}\left<M^{s}_{\theta}(t,\cdot),\varphi(t,\cdot)\right>dt
.\label{term23}
\end{align}
We turn to the left-hand side of \cref{cttau_}. Thanks to \cref{assvar} and \cref{ineqG}, $G(c)\in L^{\infty}_{loc}(\R_0^+;\LInf)$, and the dominated convergence theorem we have 
\begin{subequations}\label{inbetween}
\begin{align}
 &-\int_0^{{T}-\tau_l}\int_{\Omega}\varphi(t,x)\frac{1}{\tau_l}(c(t+\tau_l,x)-c(t,x))g(c)(t,x)\,dxdt\nonumber\\
 \geq &-\int_0^{{T}-\tau_l}\int_{\Omega}\varphi(t,x)\frac{1}{\tau_l}\left(G(c)(t+\tau_l,x)-G(c)(t,x)\right)\,dxdt\nonumber\\
 =&\int_{\tau_l}^{T}\int_{\Omega}\frac{1}{-\tau_l}\left(\varphi(t-\tau_l,x)-\varphi(t,x)\right)G(c)(t,x)\,dxdt\nonumber\\
 &+\frac{1}{\tau_l}\int_0^{\tau_l}\int_{\Omega}\varphi(t,x)G(c)(t,x)\,dxdt-\frac{1}{\tau_l}\int_{{T}-\tau_l}^{T}\int_{\Omega}\varphi(t,x)G(c)(t,x)\,dxdt\nonumber\\
 \underset{l\to\infty}{\to}&\int_0^{T}\int_{\Omega}\partial_t\varphi G(c)\,dxdt,\label{inb1}
 \end{align}
 and, similarly,
 \begin{align}
 &-\int_0^{{T}-\tau_l}\int_{\Omega}\varphi(t+\tau_l,x)\frac{1}{\tau_l}(c(t+\tau_l,x)-c(t,x))g(c)(t+\tau_l,x)\,dxdt
 \nonumber\\
 \leq &-\int_0^{{T}-\tau_l}\int_{\Omega}\varphi(t+\tau_l,x)\frac{1}{\tau_l}\left(G(c)(t+\tau_l,x)-G(c)(t,x)\right)\,dxdt\nonumber\\
 =&\int_0^{{T}-\tau_l}\int_{\Omega}\frac{1}{\tau_l}\left(\varphi(t+\tau_l,x)-\varphi(t,x)\right)G(c)(t,x)\,dxdt\nonumber\\
 &+\frac{1}{\tau_l}\int_0^{\tau_l}\int_{\Omega}\varphi(t,x)G(c)(t,x)\,dxdt-\frac{1}{\tau_l}\int_{{T}-\tau_l}^{T}\int_{\Omega}\varphi(t,x)G(c)(t,x)\,dxdt\nonumber\\
 \underset{l\to\infty}{\to}&\int_0^{T}\int_{\Omega}\partial_t\varphi G(c)\,dxdt.
\end{align}
\end{subequations}
Altogether,  \cref{cttau_,term23,inbetween} lead to
\begin{align}
&\int_0^{T}\left<M^{s}_0(t,\cdot),\varphi(t,\cdot)\right>dt\nonumber\\
\leq &\int_0^{T}\int_{\Omega}\partial_t\varphi G(c)\,dxdt-\int_0^{T}\int_{\Omega}\varphi J\cdot \nabla g(c)+\nabla\varphi\cdot J g(c)-\varphi f_c(c,v)\, g(c)\,dxdt\nonumber\\
\leq &\int_0^{T}\left<M^{s}_1(t,\cdot),\varphi(t,\cdot)\right>dt. \label{term46}
\end{align}
Combining \cref{term46,Mth,Msbnd1}, we obtain
\begin{align}
 &\left|\int_0^{T}\int_{\Omega}\partial_t\varphi G(c)\,dxdt-\int_0^{T}\int_{\Omega}\varphi J\cdot \nabla g(c)+\nabla\varphi\cdot J g(c)-\varphi f_c(c,v)\, g(c)\,dxdt\right|\nonumber\\
 \leq &(1+\chi)\Cr{CBND}(T)\int_0^{T}\left<|D^sg(c(t,\cdot))|,\varphi(t,\cdot)\right>dt. \label{term48_}
\end{align}
Since $C^1([0,T]\times\overline{\Omega};\R_0^+)\cap C_0([0,T],C_0(\overline\Omega;\R_0^+))$ is dense in $C_0([0,T],C_0(\overline\Omega;\R_0^+))$, inequality \cref{term48_} implies for $\mu_g$ from \cref{evolgc} that
\begin{subequations}\label{term49_12}
\begin{align}
&\mu_g\in (C_0([0,T],\Co))^*,\\
& |\left<\mu_g,\varphi\right>|\leq (1+\chi)\Cr{CBND}(T)\int_0^{T}\left<|D^sg(c(t,\cdot))|,\varphi(t,\cdot)\right>dt\qquad\text{for all }0\leq\varphi\in C_0([0,T];\Co).
\end{align} 
\end{subequations}
In consequence of \cref{term49_12}, for all $0\leq\zeta\in C_0([0,T])$ we have
\begin{align}
&\zeta\mu\in (C([0,T],\Co))^*,\nonumber\\
& |\left<\zeta\mu_g,\varphi\right>|\leq (1+\chi)\Cr{CBND}(T)\int_0^{T}\left<|D^sg(c(t,\cdot))|,\varphi(t,\cdot)\right>dt\qquad\text{for all }\psi\in C([0,T];\Co).
 \label{term49}
\end{align}
Combining \cref{RMKV,term49}, we conclude that 
\begin{align}
 &\mu_g\in {\cal M}_{loc}((0,T);\RM),\nonumber\\
 &|\mu_g(E,\cdot)|\leq (1+\chi)\Cr{CBND}(T)|D^sg(c)(E,\cdot)|\qquad\text{in }\RM\qquad\text{for all Borel }E\text{ such that }\overline{E}\subset (0,T).\label{est33}
\end{align}
\cref{LemRep}\cref{LemRep2} can now be applied to $\mu:=\mu_g$, $g:=(1+\chi)\Cr{CBND}(T)\|D^sg(c)\|_{\RM}$, and any closed interval $I\subset(0,T)$, 
yielding
\begin{align}
 \mu_g\in L^1_{w-*,loc}((0,\infty);\RM)\nonumber
\end{align}
and allowing to replace \cref{est33} by \cref{est33_}.
Finally, combining \cref{est33_,nDc3}, we arrive at \cref{mug}.

\end{proofpart}
\begin{proofpart}[Inequalities for the evolution of $G(c)$: part II]\label{ppII}
Our next set of inequalities for the evolution of $G(c)$ is derived from the approximating problems \cref{appr} and takes the form 
\begin{align}
  &\int_0^T \eta(t)\left<\mu_g(t,\cdot),\rho\right>dt\nonumber\\
  \leq &-\underset{l\to\infty}{\lim\inf}\int_0^T\eta\oPi_{n_l}\int_{\Omega}\rho a(c_{m_ln_l},\nabla c_{m_ln_l})\cdot\nabla g(c_{m_ln_l})\,dxdt+\int_0^T\eta\int_{\Omega} \rho z\cdot\nabla g(c)\,dxdt\nonumber\\
 &+\chi\int_0^T\eta(t)\left<D^s\int_0^{c(t,\cdot)}ug'(u)\,du\cdot\Phi(\nabla v(t,\cdot)),\rho\right>dt,\label{ineq2}
\end{align}
where $(m_l)$ and $(n_l)$ are as in {\it Step \ref{PP1}}.

Testing \cref{appr} with $\eta\rho \oPi_{n_l}g(c_{m_ln_l})$ for $l\in\N$, we estimate using   \cref{propg}, \cref{defJm}, and \cref{asstestG},  the continuity of $\Phi$ and $\nabla v$, partial integration, and the chain rule:
\begin{align}
&-\int_0^T\eta'\int_{\Omega}\rho\Lambda_{n_l}G(c_{m_ln_l})\,dxdt\nonumber\\
 =
&\int_0^T\eta\int_{\Omega}\rho\partial_t\Lambda_{n_l}G(c_{m_ln_l})\,dxdt\nonumber\\
 \leq&\int_0^T\eta\int_{\Omega}\rho\partial_t\Lambda_{n_l}c_{m_ln_l}\oPi_{n_l}g(c_{m_ln_l})\,dxdt\nonumber\\
 =&-\int_0^T\eta\oPi_{n_l}\int_{\Omega}\rho\left(\frac{1}{m_l}\nabla c_{m_ln_l}+a(c_{m_ln_l},\nabla c_{m_ln_l})\right)\cdot\nabla g(c_{m_ln_l})\,dxdt\nonumber\\
 &+\chi\int_0^T\eta\oPi_{n_l}\int_{\Omega}\rho\nabla \int_0^{c_{m_ln_l}}ug'(u)\,du\cdot \Phi(\nabla P_{n_l}v)\,dxdt\nonumber\\
 &-\int_0^T\eta\oPi_{n_l}\int_{\Omega}\nabla\rho\cdot J_{m_l}[c_{m_ln_l},P_{n_l}v] g(c_{m_ln_l})-\rho f_c(c_{m_ln_l},P_{n_l}v) g(c_{m_ln_l})\,dxdt\nonumber\\
 \leq&-\int_0^T\eta\oPi_{n_l}\int_{\Omega}\rho a(c_{m_ln_l},\nabla c_{m_ln_l})\cdot\nabla g(c_{m_ln_l})\,dxdt\nonumber\\
 &+\chi\int_0^T\eta\oPi_{n_l}\int_{\Omega}\rho\nabla \int_0^{c_{m_ln_l}}ug'(u)\,du\cdot \Phi(\nabla P_{n_l}v)\,dxdt\nonumber\\
 &-\int_0^T\eta\oPi_{n_l}\int_{\Omega}\nabla\rho\cdot J_{m_l}[c_{m_ln_l},P_{n_l}v] g(c_{m_ln_l})-\rho f_c(c_{m_ln_l},P_{n_l}v) g(c_{m_ln_l})\,dxdt.\label{est31}
\end{align}
 Combining   \cref{LaGconveL2}, \cref{convDc2}, \cref{convJ}, \cref{PifconveLinf}, and  \cref{convPn} and using the continuity of $\Phi$ and $\nabla v$ and the compensated compactness, we pass to the limit superior on both sides of \cref{est31} and obtain
 \begin{align}
&-\int_0^T\eta'\int_{\Omega}\rho G(c)\,dxdt\nonumber\\
 \leq&-\underset{l\to\infty}{\lim\inf}\int_0^T\eta\oPi_{n_l}\int_{\Omega}\rho a(c_{m_ln_l},\nabla c_{m_ln_l})\cdot\nabla g(c_{m_ln_l})\,dxdt\nonumber\\
 &+\chi\int_0^T\eta(t)\left<D\int_0^{c(t,\cdot)}ug'(u)\,du\cdot\Phi(\nabla v(t,\cdot)),\rho\right>dt-\int_0^T\eta\int_{\Omega}\nabla\rho\cdot J g(c)-\rho f_c(c,v) g(c)\,dxdt.\label{term50}
 \end{align}
Further, on the right-hand-side of \cref{term50} we make use of   \cref{nDc2str}, \cref{nDc3}, and \cref{chaingc} which yields
\begin{align}
&-\int_0^T\eta'\int_{\Omega}\rho G(c)\,dxdt\nonumber\\
 \leq&-\underset{l\to\infty}{\lim\inf}\int_0^T\eta\oPi_{n_l}\int_{\Omega}\rho a(c_{m_ln_l},\nabla c_{m_ln_l})\cdot\nabla g(c_{m_ln_l})\,dxdt+\int_0^T\eta\int_{\Omega} \rho z\cdot\nabla g(c)\,dxdt\nonumber\\
 &-\int_0^T\eta\int_{\Omega}\rho J\cdot\nabla g(c)+\nabla\rho\cdot J g(c)-\rho f_c(c,v) g(c)\,dxdt\nonumber\\
 &+\chi\int_0^T\eta(t)\left<D^s\int_0^{c(t,\cdot)}ug'(u)\,du\cdot\Phi(\nabla v(t,\cdot)),\rho\right>dt.\label{est35}
 \end{align}
 Combining \cref{evolgc} and \cref{est35}, we arrive at   \cref{ineq2}.

\end{proofpart}

\begin{proofpart}[Identifying $z$] Our first application of the inequalities established in {\it Steps \ref{ppI}} and {\it\ref{ppII}} is to the proof of 
\begin{align}
z=a(c,\nabla c)
\qquad\text{ a.e. in }\R_0^+\times\Omega
,\label{zincpos}                                            \end{align}
so that 
\begin{align}
 J=a(c,\nabla c)-\chi c\Phi(\nabla v)\qquad\text{ a.e. in }(0,\infty)\times\Omega\nonumber%
\end{align}
in the limit equation \cref{limeq}.
As is customary in such cases, we couple the Minty-Browder method %
 with compactness.    
Let 
\begin{align}
 y\in\R^d.\nonumber%
\end{align}
Due to the 
 monotonicity of $a$ with respect to its second argument, i.e. property \cref{amon}, it holds that  
\begin{align}
 0\leq&\int_0^T\eta\oPi_{n_l}\int_{\Omega}\rho\left(a(c_{m_ln_l},\nabla c_{m_ln_l}\right)-a(c_{m_ln_l},y))\cdot(\nabla c_{m_ln_l}-y)\,dxdt.\label{pos1}
 \end{align}
 Obviously, the right-hand side of \cref{pos1} can be decomposed as follows:
 \begin{align}
 &\int_0^T\eta\oPi_{n_l}\int_{\Omega}\rho\left(a(c_{m_ln_l},\nabla c_{m_ln_l}\right)-a(c_{m_ln_l},y))\cdot(\nabla c_{m_ln_l}-y)\,dxdt\nonumber\\
 =&\int_0^T\eta\oPi_{n_l}\int_{\Omega}\rho a(c_{m_ln_l},\nabla c_{m_ln_l})\cdot\nabla c_{m_ln_l}\,dxdt\nonumber\\
 &-\int_0^T\eta\oPi_{n_l}\int_{\Omega}\rho\left(a(c_{m_ln_l},\nabla c_{m_ln_l}\right)-a(c_{m_ln_l},y))\cdot y\,dxdt \nonumber\\
 &-\int_0^T\eta\oPi_{n_l}\int_{\Omega}\rho a(c_{m_ln_l},y)\cdot\nabla c_{m_ln_l}\,dxdt. \label{RHSpos1}
\end{align}
We consider the limit behaviour of each of the terms in the decomposition in \cref{RHSpos1} separately. 
To handle the first term, we apply \cref{ineq2} for $g(c):=c$, so that
\begin{align}
 &\underset{l\to\infty}{\lim\inf}\int_0^T\eta\oPi_{n_l}\int_{\Omega}\rho a(c_{m_ln_l},\nabla c_{m_ln_l})\cdot\nabla c_{m_ln_l}\,dxdt\nonumber\\
 \leq&\int_0^T\eta\int_{\Omega} \rho z\cdot \nabla c\,dxdt
 +\chi\int_0^T\eta(t)\left<D^sc(t,\cdot)\cdot\Phi(\nabla v(t,\cdot)),\rho\right>dt-\int_0^T \eta(t)\left<\mu_{id}(t,\cdot),\rho\right>dt.
 \label{term6_}
\end{align}
Combining \cref{est33_,Phi1,nDc4}, we estimate the singular measures on the right-hand side of \cref{term6_} and conclude that
\begin{align}
 &\underset{l\to\infty}{\lim\inf}\int_0^T\eta\oPi_{n_l}\int_{\Omega}\rho a(c_{m_ln_l},\nabla c_{m_ln_l})\cdot\nabla c_{m_ln_l}\,dxdt\nonumber\\
 \leq&\int_0^T\eta\int_{\Omega} \rho z\cdot \nabla c\,dxdt+\C(T)\int_0^T\eta(t)\left<|D^s c(t,\cdot)|,\rho\right>dt.
 \label{term6}
\end{align}
With the help of  \cref{PiconvgvL1,conva} we obtain
\begin{align}
 \int_0^T\eta\oPi_{n_l}\int_{\Omega}\rho\left(a(c_{m_ln_l},\nabla c_{m_ln_l}\right)-a(c_{m_ln_l},y))\cdot y\,dxdt
 \underset{l\rightarrow\infty}{\rightarrow}&\int_0^T\eta\int_{\Omega}\rho(z-a(c,y))\cdot y\,dxdt.\label{conv1}
\end{align}
Due to the chain rule and \cref{convDc2} applied to $g(c):=\int_0^c a(u,y)\,du$   it holds that
\begin{align}
 \oPi_{n_l}\left(a(c_{m_ln_l},y)\cdot\nabla c_{m_ln_l}\right)=&\oPi_{n_l}\nabla\cdot\int_0^{c_{m_ln_l}}a(u,y)\,du\nonumber\\
 \underset{l\rightarrow\infty}{\overset{*}{\rightharpoonup}}&\nabla\cdot\int_0^c a(u,y)\,du\qquad\text{in }{\cal M}_{loc}(\R_0^+;\RM).\label{conv2_}
\end{align}
Utilising 
\cref{chaingc} and the continuity of $a$, we find that 
\begin{subequations}\label{acp1}
\begin{alignat}{3}
 &\left(\nabla\cdot\int_0^{c(t,\cdot)} a(u,y)\,du\right)^{ac}
 =(a(c,y)\cdot\nabla c)(t,\cdot)&&\qquad\text{in }\Omega\text{ for a.a. }t\in(0,T),\\
 &\left|\left(\nabla\cdot\int_0^{c(t,\cdot)} a(u,y)\,du\right)^{s}\right|\leq T\Cr{CBND}(T)|D^s c(t,\cdot)|&&\qquad\text{in }\RM\text{ for a.a. }t\in(0,\infty).
\end{alignat}
\end{subequations}
Combining \cref{conv2_,acp1}, we obtain
\begin{align}
 &\underset{l\to\infty}{\lim}\int_0^T\eta\oPi_{n_l}\int_{\Omega}\rho\left(a(c_{m_ln_l},y)\cdot\nabla c_{m_ln_l}\right)\,dxdt\nonumber\\
 \leq& \int_0^T\eta \int_{\Omega}\rho a(c,y)\cdot\nabla c\,dxdt+T\Cr{CBND}(T)\int_0^T\left<|D^s c(t,\cdot)|,\rho\right>dt.\label{conv2}
\end{align}
Collecting  \cref{pos1,RHSpos1,conv1,conv2,term6}, we finally arrive at
the inequality
\begin{align}
\int_0^T\eta\int_{\Omega}\rho(z-a(c,y))\cdot(\nabla c-y)\,dxdt+T\Cl{Clast}\int_0^T\eta(t)\left<|D^sc(t,\cdot)|,\rho\right>dt\geq0. \label{limmeas}
\end{align}
Given that $\rho$ and $\eta$ are  arbitrary, \cref{limmeas} implies 
\begin{align}
 ((z-a(c,y))\cdot(\nabla c-y))(t,\cdot)+T\Cr{Clast}|D^sc(t,\cdot)|\geq0\qquad\text{in }\RM\text{ for a.a. }t\in(0,T).\label{measuresum}
\end{align}
The ($d$-dimensional) Lebesgue measure and $|D^sc(t,\cdot)|$ are disjoint on $\Omega$. Hence, \cref{measuresum} yields
\begin{align}
 (z-a(c,y))\cdot(\nabla c-y)\geq0\qquad\text{a.e. in }(0,\infty)\times\Omega\text{ for all }y\in\R^d.\label{z1_}
\end{align}
Replacing $\R^d$ by the countable set $\Q^d$, we conclude with \cref{z1_} that
\begin{align}
 (z-a(c,y))\cdot(\nabla c-y)\geq0\qquad\text{for all }y\in\Q^d\text{ a.e. in }(0,\infty)\times\Omega,\label{zineq0}
\end{align} 
Since $\Q^d$ is dense in $\R^d$ and $a$ is continuous in the second variable, \cref{zineq0} implies
\begin{align}
 (z-a(c,y))\cdot(\nabla c-y)\geq0\qquad\text{for all }y\in\R^d\text{ a.e. in }(0,\infty)\times\Omega.\label{zineq}
\end{align}
Combining \cref{zineq} and a standard argument based on the continuity of $a$ with respect to the second variable, we arrive at \cref{zincpos}. 
 \end{proofpart}
 
 \begin{proofpart}[Entropy inequalities]
In order to establish the entropy inequalities \cref{entropyin}, we choose $g$ of the form
 \begin{align}
  g:=T_{m,M}\label{lscdg}
 \end{align}
 for $0\leq m< M<\infty$. This function is increasing, globally Lipschitz, and  satisfies
 \begin{subequations}
 \begin{align}
  &g'(u)=\begin{cases}
         1&\text{for }u\in(m,M),\\
         0&\text{for }u\in\R_0^+\backslash[m,M],
        \end{cases}\label{derg}\\
&ug'(u)=\frac{1}{2}(g^2(u))'\qquad\text{for }u\in\R_0^+\backslash\{m,M\}.\label{derg2}
 \end{align}
 \end{subequations}
We also make use of the auxiliary function 
\begin{align*}
 L(u,\xi):=u\sqrt{u^2+|\xi|^2}-u^2\qquad\text{for }u\geq0,\ \xi\in\R^d.
\end{align*}
One readily verifies that \footnotetext[1]{This means that $L$ is a Lagrangian which $a$ is associated with. }
\begin{subequations}\label{propL}
\begin{alignat}{3}
&L\in C^1(\R_0^+\times \R^d)\cap C^2((0,\infty)\times \R^d),&&\label{Lconti}\\
 &L(u,0)=0&&\qquad\text{for }\xi\in\R^d,\label{L0}\\
&0< L(u,\xi)< u|\xi|&&\qquad\text{for }u>0,\ \xi\in\R^d\backslash\{0\},\label{Lsublin}\\
&L(u,\cdot)\text{ is convex in }\R^d&&\qquad\text{for } u\geq0,\label{Lconvex}\\
&\nabla_{\xi}L(u,\xi)=a(u,\xi)\ \footnotemark[1],&&\qquad\text{for }u\geq0,\ \xi\in\R^d,\label{dL}\\
& a(u,\xi)\cdot \xi\geq L(u,\xi)&&\qquad\text{for }u\geq0,\ \xi\in\R^d,\label{conc}\\
&L^{\infty}(u,\xi)=u|\xi|=(a(u,\cdot)\cdot(\cdot))^{\infty}(\xi)&&\qquad\text{for }u\geq0,\ \xi\in\R^d.\label{recessL}
\end{alignat}
\end{subequations}  
 Since $a(\cdot,0)\equiv0$, we have in consequence of \cref{derg,conc} and the chain rule that  
 \begin{align}
  a(u,\nabla u)\cdot \nabla g(u)=&a(u,\nabla g(u))\cdot \nabla g(u)\nonumber\\
  \geq &L(g(u),\nabla g(u))\qquad\text{a.e. in }\Omega \text{ for all }u\in\WOne.\label{compaL}
 \end{align}
 Consider now 
 \begin{align}
  f (x,w,\xi):=\rho(x)L\left(w_+,\xi\right)\qquad\text{for }(x,w,\xi)\in\Omega\times\R\times\R^d\nonumber%
 \end{align} 
  and the corresponding relaxation 
  functional \cite{DalMaso} for $0\leq w\in\BV\cap\LInf$:
  \begin{align*}
   {\cal R}_f(w):=&\int_{\Omega}f(x,w(x),\nabla w(x))\,dx+\int_{\Omega\backslash S_w}f^{\infty}(x,\widetilde w(x),\nu_w(x))\,d|D^c w|(x)\nonumber\\
   &+\int_{J_u} \frac{\int_{w^-(x)}^{w^+(x)}f^{\infty}(x,z,\nu_w(x))\,dz}{w^+(x)-w^-(x)}\,d |D^j w|(x)\nonumber\\
   =&\int_{\Omega}\rho L\left(w,\nabla w\right)\,dx+\frac{1}{2}\int_{\Omega}\rho \,d|D^s w^2|.
\end{align*}  
Here we used \cref{recessL} and the chain rule for BV functions to simplify the expression for ${\cal R}_f$. 
By the assumptions on $\rho$ in \cref{rhoeta} and properties \cref{Lconti,Lconvex,Lsublin} of $L$  the requirements of  \cite[Theorem 1.1]{CiccoFuscoVerde} are satisfied for $f$. Therefore, ${\cal R}_f$
is lower semi-continuous in $\BV$ equipped with the $\LOne$ norm. Together with \cref{PiconvgvL1ae} this implies
\begin{align}
 &\underset{l\to\infty}{\lim\inf}\int_{\Omega}\rho L(\oPi_{n_l}g(c_{m_ln_l}),\nabla \oPi_{n_l}g(c_{m_ln_l}))\,dx(t,\cdot)\nonumber\\
 \geq &\int_{\Omega}\rho L\left(g(c),\nabla g(c)\right)\,dx(t,\cdot)+\left<\frac{1}{2}|D^s g^2(c(t,\cdot))|,\rho\right>\qquad \text{ for a.a. }t\in(0,T).\label{est32aa}
\end{align}
Consequently, Fatou's lemma combined with \cref{nDc4} and \cref{est32aa} allows to conclude that 
\begin{align}
 &\underset{l\to\infty}{\lim\inf}\int_0^T\eta\int_{\Omega}\rho L(\oPi_{n_l}g(c_{m_ln_l}),\nabla \oPi_{n_l}g(c_{m_ln_l}))\,dxdt\nonumber\\
 \geq &\int_0^T\eta\,\underset{l\to\infty}{\lim\inf}\int_{\Omega}\rho L(\oPi_{n_l}g(c_{m_ln_l}),\nabla \oPi_{n_l}g(c_{m_ln_l}))\,dxdt\nonumber\\
 \geq &\int_0^T\eta\int_{\Omega}\rho L\left(g(c),\nabla g(c)\right)\,dxdt+\int_0^T\eta(t)\left<\frac{1}{2}|D^s g^2(c(t,\cdot))|,\rho\right>\,dt.\label{est32}
\end{align}
We use \cref{compaL,est32} to estimate the first term on right-hand side of \cref{ineq2} and then \cref{derg2}, yielding
\begin{align}
  \int_0^T \eta(t)\left<\mu_g(t,\cdot),\rho\right>dt
  \leq &\int_0^T\eta\int_{\Omega}\rho \left(-L\left(g(c),\nabla g(c)\right)+z\cdot\nabla g(c)\right)\,dx\nonumber\\
 &+\int_0^T\eta(t)\left<-\frac{1}{2}|D^s g^2(c(t,\cdot))|+\frac{\chi}{2}D^s g^2(c(t,\cdot))\cdot\Phi(\nabla v(t,\cdot)),\rho\right>dt.\label{ineq2_}
\end{align}
Since $\rho$ and $\eta$ are  arbitrary and $\mu_g(t,\cdot)$ is singular for a.a. $t$, \cref{ineq2_} implies 
 \begin{align}
  \mu_g(t,\cdot)\leq -\frac{1}{2}|D^s g^2(c(t,\cdot))|+\frac{\chi}{2}D^s g^2(c(t,\cdot))\cdot\Phi(\nabla v(t,\cdot))\qquad\text{in }\RM\text{ a.e. in }(0,\infty).\label{ineq5}
 \end{align}
Finally, combining  \cref{evolgc} and \cref{ineq5} and recalling \cref{lscdg}, we obtain the entropy inequality \cref{EstAbove}. As to \cref{EstBelow}, it is a special case of the lower bound for $\mu_g$ provided in \cref{mugen}.
 
 \end{proofpart}

 \begin{proofpart}[Stability] 
 In order to obtain \cref{contic}, we  aim to pass to the limit along suitable subsequences in \cref{contice}.
We need some preparation. 

 As is well-known, each $0\leq c_0\in\LInf$ can be approximated in the $\LOne$-norm by an increasing sequence of nonnegative simple functions from a countable set ${\cal C}_0$. Further, by exploiting standard separability results, one readily confirms that each $v$ that belongs to 
 \begin{align}
   C(\R_0^+;C^1_{\nu}(\overline{\Omega}))\cap L^1_{loc}(\R_0^+;W^{2,1}(\Omega))\nonumber
 \end{align}
 can be approximated by a sequence $(v_k)$ from a countable subset ${\cal C}_v$ of this space in the following way: for all $T\in\N$
 \begin{align*}
  &\|v_k-v\|_{C([T,T+1];C^1_{\nu}(\overline{\Omega}))}\underset{k\to\infty}{\to}0,\\
  &\|v_k-v\|_{L^1_{loc}((T,T+1);W^{2,1}(\Omega))}\underset{k\to\infty}{\to}0,\\
  &\|v_k\|_{L^{\infty}((0,T);\LInf)}\leq\|v\|_{L^{\infty}((0,T);\LInf)}\qquad\text{for all }k\in\N.
 \end{align*}
For $(c_0,v)$ that satisfies \cref{c0vreg} let $(m_l(c_0,v))$ and $(n_l(c_0,v))$ denote two corresponding sequences as constructed in {\it Step \ref{PP1}} of the present proof. 
Since ${\cal C}_0\times {\cal C}_v$ is countable, the diagonal argument allows to choose these sequences in the following way: 
\begin{alignat*}{3}
 &m_l(c_0,v)=m_l^{(*)}&&\qquad\text{for all }l\in\N,\ (c_0,v)\in {\cal C}_0\times {\cal C}_v,\\
 &n_l(c_0,v)=n_l^{(*)}&&\qquad\text{for all }l\in\N,\ (c_0,v)\in {\cal C}_0\times {\cal C}_v,
\end{alignat*}
and 
\begin{alignat*}{3}
 &(m_l(c_0,v))\text{ is a subsequence of }\left(m_l^{(*)}\right)&&\qquad\text{for all }(c_0,v)\text{ as in }\cref{c0vreg},\\
 &(n_l(c_0,v))\text{ is a subsequence of }\left(n_l^{(*)}\right)&&\qquad\text{for all }(c_0,v)\text{ as in }\cref{c0vreg},
\end{alignat*}
where $\left(m_l^{(*)}\right)$ and $\left(n_l^{(*)}\right)$ are some fixed sequences.

Now we are ready to prove \cref{contic} for the case when $\left(c_0^{(1)},v^{(1)}\right)$ satisfies \cref{c0vreg} and $\left(c_0^{(2)},v^{(2)}\right)\in {\cal C}_0\times {\cal C}_v$. Choosing 
$$m:=m_l\left(c_0^{(1)},v^{(1)}\right)\qquad\text{and}\qquad n:=n_l\left(c_0^{(1)},v^{(1)}\right)$$
in \cref{contice} for sufficiently large $l\in\N$, so that $n_l\left(c_0^{(1)},v^{(1)}\right)>N$ for $N$ as in \cref{LemSolAppr}\cref{itemunn}, we can pass to the limit in \cref{contice} using  \cref{C20conv}, \cref{XTM}, \cref{F4}, \cref{MT}, and \cref{convL1ae} and the elementary fact that  $\exp_{n,C}(t)\to e^{tC}$ as $n\to\infty$. This yields \cref{contic} for almost all $t\in(0,T)$. Taking into account $c^{(1)}-c^{(2)}\in C_{w}([0,T];\LOne)$ and the lower weak semicontinuity of norms, we conclude that \cref{contic} is valid for all $t\in[0,T]$.
With this knowledge at hand we turn to the general data
 $\left(c_0^{(i)},v^{(i)}\right)$ satisfying  \cref{c0vreg} for $i\in\{1,2\}$. Approximating each component in the way as described at the beginning of this Step of the proof by the respective sequences $\left(c_{k0}^{(i)},v_k^{(i)}\right)$, we obtain for the corresponding solutions, $c^{(1)}$ and $c^{(2)}$, 
\begin{align}
 &\left\|c^{(1)}-c^{(2)}\right\|_{\LOne}\nonumber\\
 \leq&\underset{k\to\infty}{\lim\inf}\left(\left\|c_k^{(1)}-c_k^{(2)}\right\|_{\LOne}+\left\|c_k^{(1)}-c^{(1)}\right\|_{\LOne}+\left\|c_k^{(2)}-c^{(2)}\right\|_{\LOne}\right)\nonumber\\
 \leq&e^{T\underset{(0,M(T))^2}{\max}\partial_cf}\underset{k\to\infty}{\lim}\left(\left\|c_{k0}^{(1)}-c_{k0}^{(2)}\right\|_{\LOne}+\left\|c_{k0}^{(1)}-c_{0}^{(1)}\right\|_{\LOne}+\left\|c_{k0}^{(2)}-c_{0}^{(2)}\right\|_{\LOne}\right)\nonumber\\
&+\Cr{C361}(M(T))\nonumber\\ &\ \ \ \cdot\underset{k\to\infty}{\lim}\left(\left\|v_k^{(1)}-v_k^{(2)}\right\|_{C([0,T];C^1(\overline\Omega))}+\left\|v_k^{(1)}-v^{(1)}\right\|_{C([0,T];C^1(\overline\Omega))}+\left\|v_k^{(2)}-v^{(2)}\right\|_{C([0,T];C^1(\overline\Omega))}\right)\nonumber\\
&+\Cr{C361}(M(T))\nonumber\\ &\ \ \ \cdot\underset{k\to\infty}{\lim}\left(\left\|v_k^{(1)}-v_k^{(2)}\right\|_{L^1((0,T);W^{2,1}(\Omega))}+\left\|v_k^{(1)}-v^{(1)}\right\|_{L^1((0,T);W^{2,1}(\Omega))}+\left\|v_k^{(2)}-v^{(2)}\right\|_{L^1((0,T);W^{2,1}(\Omega))}\right)\nonumber\\
=&e^{T\underset{(0,M(T))^2}{\max}\partial_cf}\left\|c_0^{(1)}-c_0^{(2)}\right\|_{\LOne}\nonumber\\
&+\Cr{C361}(M(T))\left(\left\|v^{(1)}-v^{(2)}\right\|_{C([0,T];C^1(\overline\Omega))}+\left\|v^{(1)}-v^{(2)}\right\|_{L^1((0,T);W^{2,1}(\Omega))}\right),\nonumber
\end{align}
as required. This completes the proof of \cref{thmSolEq1}.

\end{proofpart}

\end{proof}
Several remarks are in order. 
\begin{Remark}\label{RemLB}
 Since {\it Step \ref{ppI}} of the above proof relies on the weak formulation \cref{limeq_} alone, the lower bound \cref{EstBelow} holds for every weak solution. 
\end{Remark}

\begin{Remark}\label{RemMethod}
 Let us compare the way we have handled {the 
  FL} reaction-diffusion-taxis equation \cref{flc} with the treatment of a FL (reaction-)diffusion equation such as presented  in {\cite{ACM021,ACM02,ACM05,ACM05Cauchy,ACMM06,ACM08,ACMM10} and other works}. These are the main differences.
 \begin{enumerate}
 \item Due to the presence of   taxis and a general {source} term, bounds for $\|c\|_{\LInf}$ had to be derived using a method that is common in analysis of taxis systems. In the present case, $c$ is, in general, only bounded over finite time intervals. 
 \item {In the proof of \cref{Lemnablac} the equation has been tested with $\ln(\max\{c,{m}\})$ rather than $T_{{m},M}(c)$, where $0<{m}<M<\infty$. This has led to the improved regularity 
 \begin{align*}
 c\in L_{w,loc}^1(\R_0^+;\BV)
 \end{align*}
compared to 
\begin{align*}
 T_{{m},M}(c)\in L_{w,loc}^1(\R_0^+;\BV)\qquad \text{for }0<{m}<M<\infty\qquad \text{(\cite{ACM021,ACM02,ACM05,ACM05Cauchy,ACMM06,ACM08,ACMM10}, etc.)}.
\end{align*}
 }
  \item We chose not to study the non-regularised, i.e. without the inclusion of operator $1/m\Delta$ in \cref{ellpr} (recall \cref{DefBmn}), disrectisations of \cref{IBVPc}. As a result, we cannot avail of the mild solution  typically obtained in the limit as $n\to\infty$ by the theory developed in  \cite{CranLigg}. Instead, our solution, as constructed in {\it Step \ref{PP1}} of the above proof, is a limit of a subsequence $(n_l)$ that may depend on $c_0$ and $v$. This mostly affected the proof of the stability property \cref{contic}, making it more involved than it would have been otherwise. 
  \item {Bypassing the  mild solution has further led to a weaker continuity in time, so that only 
  \begin{align*}
   c\in C_w(\R_0^+;\LOne)
  \end{align*}
  has been established and not
  \begin{align*}
   c\in C(\R_0^+;\LOne)\qquad \text{(\cite{ACM021,ACM02,ACM05,ACM05Cauchy,ACMM06,ACM08,ACMM10}, etc.)}.
  \end{align*}
  }
  \item   
  On the whole, more use of compactness {has been}  necessary by the presence of {nonlinear} terms other than diffusion.
  \item Interpolation operators and their properties (presented in  \cref{AppendixInter}) {have been used extensively}. It {has} allowed to embed the elliptic approximations into the framework of continuous-time equations. 
  \item {The use of} several lemmas about functions and measures with values in Banach spaces (presented in \cref{AppVM}) {has} further streamlined the analysis. 
  \item {We have not dealt with uniqueness of entropy solutions (for fixed $v$, that is). It appears that it could be obtained in the very same way as in \cite{ACM05}, after verifying conditions on the time derivative of $c$ as well as its spatial boundary trace akin to those included in \cite[Definition 6]{ACM05} but not in \cref{Defsol1}. Since the arguments in that proof do not require $C(\R_0^+;\LOne)$, our solution must then coincide with the unique entropy solution of that kind. In particular, for $v\equiv const$, we regain the mild solution for the reaction-diffusion equation.
  }
 \end{enumerate}

\end{Remark}

\section{Strong well-posedness of \texorpdfstring{\cref{IBVPv}}{} for fixed \texorpdfstring{$c$}{c}}\label{Secv}
Unlike the non-trivial flux-limited equation \cref{flc}, the   semilinear parabolic PDE \cref{flv} can be handled using a classical theory. Similarly to \cref{thmSolEq1}, we provide a solvability result that is  formulated in terms of a solution mapping. This time, however, we have chosen not to track the dependence upon the initial value since the equation is a standard one. 
\begin{Lemma}[Strong solutions to \cref{IBVPv}]\label{thmSolEq2} Let $D_v>0$ and assumptions \cref{fv,v0bnd} be satisfied. 
Then there exists a mapping $\Ss_2$ that maps any $c$ that satisfies  \cref{regcinv} to the unique $v$ such that $(c,v)$ solves \cref{IBVPv} in terms of \cref{Defsol2} and for all $T>0$ 
\begin{align}
&\|v\|_{L^{\infty}((0,T);\LInf)}\leq \C\left(T,\|c\|_{L^{\infty}((0,T);\LInf)}\right).\label{estvbnd}
\end{align}
Mapping $\Ss_2$ is locally Lipschitz  continuous in the following sense: if $v^{(i)}:=\Ss_2\left(c^{(i)}\right)$, then for all $T>0$
 \begin{subequations}\label{conticv}
 \begin{align}
  &\left\|v^{(1)}-v^{(2)}\right\|_{W^{1,p_0}((0,T);L^{p_0}(\Omega))}+\left\|v^{(1)}-v^{(2)}\right\|_{L^{p_0}((0,T);W^{2,p_0}(\Omega))}
  \nonumber\\
  \leq&\C\left(T,M\left[c^{(1)},c^{(2)}\right](T)\right) \left\|c^{(1)}-c^{(2)}\right\|_{L^{p_0}((0,T);L^{p_0}(\Omega))},
 \end{align}
  where 
\begin{align}
  M\left[c^{(1)},c^{(2)}\right](T):=\left\|\left(c^{(1)},c^{(2)}\right)\right\|_{(L^{\infty}((0,T);\LInf))^2}.
 \end{align}
 \end{subequations}

\end{Lemma}

\begin{proof}[{\bf Proof of} \cref{thmSolEq2}]
\begin{proofpart}[Existence, uniqueness, and bounds]
Thanks to the assumptions on $c$ and $f_v$, estimate \cref{estvbnd} as well as the non-negativity hold a priori by the maximum principle. Standard theory of semilinear parabolic equations  provides the existence and uniqueness of a global weak solution $v$ to \cref{IBVPv} that satisfies \cref{estvbnd}.  Exploiting the maximal Sobolev regularity for linear parabolic equations in $L^p$ (e.g. as stated in \cite[Chapter 4 Remark 4.10.9$(c)$]{Amannbook95}) yields \cref{regv}.
\end{proofpart}
\begin{proofpart}[Stability]
Consider solutions $v^{(1)}$ and $v^{(2)}$ that correspond to some $c^{(1)}$ and $c^{(2)}$, respectively, and have the same initial value. 
Using \cite[Chapter 4 Remark 4.10.9$(c)$]{Amannbook95}, \cref{estvbnd}, and $f_v\in C^1$, we obtain for $0<t<T<\infty$
\begin{align}
 &\left\|\left(v^{(1)}-v^{(2)}\right)(t,\cdot)\right\|_{L^{p_0}(\Omega)}^{p_0}+\left\|v^{(1)}-v^{(2)}\right\|_{W^{1,p_0}((0,t);L^{p_0}(\Omega))}^{p_0}+\left\|v^{(1)}-v^{(2)}\right\|_{L^{p_0}((0,t);W^{2,p_0}(\Omega))}
  ^{p_0}\nonumber\\
 \leq& \C(T)\int_0^t\left\|f_v\left(c^{(1)},v^{(1)}\right)-f_v\left(c^{(2)},v^{(2)}\right)\right\|_{L^{p_0}(\Omega)}^{p_0}\,ds\nonumber\\
 \leq&\Cl{C21}\left(T,M\left[c^{(1)},c^{(2)}\right](T)\right)\left(\int_0^t\left\|v^{(1)}-v^{(2)}\right\|_{L^{p_0}(\Omega)}^{p_0}\,ds+\int_0^t\left\|c^{(1)}-c^{(2)}\right\|_{L^{p_0}(\Omega)}^{p_0}\,ds\right).\label{est21}
\end{align}
Omitting the second and third summands on the left-hand side of \cref{est21} and using Gronwall's lemma, we arrive at
 \begin{align}
  \left\|v^{(1)}-v^{(2)}\right\|_{L^{\infty}((0,T),L^{p_0}(\Omega))}^{p_0}\leq \C\left(T,M\left[c^{(1)},c^{(2)}\right](T)\right) \left\|c^{(1)}-c^{(2)}\right\|_{L^{p_0}((0,T);L^{p_0}(\Omega))}^{p_0}.\label{est22}
 \end{align}
 Altogether, \cref{est21,est22} yield \cref{conticv}.

\end{proofpart}

\end{proof}

\section{Existence of solutions to system \texorpdfstring{\cref{IBVPc}-\cref{IBVPv}}{}}\label{SecSchauder} 
\cref{thmSolEq1,thmSolEq2} provide solvability of \cref{IBVPc}  for given $v$ and  \cref{IBVPv} for given $c$, respectively. With these results at hand, we can now prove our main \cref{mainthm} on the solvability of the complete system  \cref{IBVPc}-\cref{IBVPv}.
\begin{proof}[{\bf Proof of}  \cref{mainthm}]
\begin{proofpart}[Reduction to local existence]
Observe that it  suffices to establish the existence of a joint solution to \cref{IBVPc}-\cref{IBVPv} on a finite time interval of a fixed length, e.g. on $[0,1]$. This is because the system is autonomous and we are looking for solutions that due to \cref{ClassC_} and \cref{contiv} satisfy 
\begin{align*}
 (c,v)(1,\cdot)\in\left(\LInf,W^{2\left(1-\frac{1}{p_0}\right),p_0}_{\nu}\right),
\end{align*}
i.e. have the value at $t=1$ with the same regularity as prescribed at $t=0$, so that this value can serve as an initial value for the time interval $[1,2]$ and so on. Thus, a solution for all times $t\in\R_0^+=\underset{N\in\N_0}{\cup}[N,N+1]$ can be build recursively. It is evident that it  necessarily satisfies \cref{regc,regv}.
\end{proofpart}
\begin{proofpart}[Fixed point argument]
We aim at applying  Schauder's fixed point theorem in order to solve \cref{IBVPc}-\cref{IBVPv} for $t\in[0,1]$. 
 Let 
 \begin{align*}
 &{\cal V}:=L^1((0,1);\LOne),\\
 &\X :=\left\{u\in L^{\infty}((0,1);\LInf):\ \|u\|_{L^{\infty}((0,1);\LInf)}\leq \Cr{CBND}(1)\right\},    \nonumber  \\
 &\Ss :\X \to \X ,\qquad \Ss (c):=\Ss_1(c_0,\Ss_2(c))\text{ for }c\in \X .
 \end{align*}
 By \cref{thmSolEq1,thmSolEq2}, mapping $\Ss $ is well-defined. Its fixed points are local solutions of the kind we need. To complete the proof, we verify the assumptions of Schauder's theorem for  $\Ss $ in $\X \subset{\cal V}$.
 
 Set $\X $ is a closed bounded convex subset of ${\cal V}$.  By construction, $\Ss (\X )$ consists of entropy solutions to the IBVP \cref{IBVPc} that belong to  the {\it closure} of 
 \begin{align}
  &\left\{c\in \X:\  c\in L^1((0,1);\WOne)\cap W^{1,1}((0,1);W^{-1,\infty}(\Omega))\right.\nonumber\\
  &\left. \phantom{c\in \X:\ \ }\text{ and }  \|c\|_{L^1((0,1);\WOne)}+\|c\|_{W^{1,1}((0,1);W^{-1,\infty}(\Omega))}\leq \C\right\}.\nonumber
 \end{align}
 The set is precompact in ${\cal V}$ and, hence, in its closed subset $\X $. This is a consequence of a version of the Lions-Aubin lemma \cite[Section 8 Corollary 4]{Simon}.  Consequently, $\Ss (\X )$ is precompact in ${\cal V}$.
 It remains to verify the continuity of $\Ss $. 
Let $c^{(i)}:=\Ss \left(\bar c^{(i)}\right)$ and $v^{(i)}:=\Ss_2\left(\bar c^{(i)}\right)$ for $i\in\{1,2\}$. 
Combining   \cref{contic,conticv,cbound,estvbnd}, embeddings \cref{embeddings},   and the interpolation inequality for $L^p$ spaces,
we estimate as follows:
 \begin{align}
 \left\|c^{(1)}-c^{(2)}\right\|_{L^1((0,1);\LOne)}
 \leq&\C\left(\left\|v^{(1)}-v^{(2)}\right\|_{C([0,1];C^1(\overline\Omega))}+\left\|v^{(1)}-v^{(2)}\right\|_{L^1((0,1);W^{2,1}(\Omega))}\right)\nonumber\\
 \leq&\C\left(\left\|v^{(1)}-v^{(2)}\right\|_{W^{1,p_0}((0,T);L^{p_0}(\Omega))}+\left\|v^{(1)}-v^{(2)}\right\|_{L^{p_0}((0,T);W^{2,p_0}(\Omega))}\right)\nonumber\\
 \leq &\C\left\|\bar c^{(1)}-\bar c^{(2)}\right\|_{L^{p_0}((0,1);L^{p_0}(\Omega))}\nonumber\\
 \leq &\C\left\|\bar c^{(1)}-\bar c^{(2)}\right\|_{L^{\infty}((0,1);\LInf)}^{1-\frac{1}{p_0}}\left\|\bar c^{(1)}-\bar c^{(2)}\right\|_{L^1((0,1);\LOne)}^{\frac{1}{p_0}}\nonumber\\
 \leq &\C\left\|\bar c^{(1)}-\bar c^{(2)}\right\|_{L^1((0,1);\LOne)}^{\frac{1}{p_0}}.\nonumber%
 \end{align}
This means that $\Ss $ is a (H\"older) continuous operator in $\X $.

\end{proofpart}
\end{proof}

\section*{Acknowledgement}
\addcontentsline{toc}{section}{Acknowledgement}
\begin{itemize}
\item {The author expresses her gratitude to the anonymous reviewers for their helpful comments that  contributed to the improvement of the paper.}
\item 
The author was supported by the Engineering and Physical Sciences Research Council [grant number
EP/T03131X/1].
\item For the purpose of open access, the author has applied a Creative Commons Attribution (CC BY) licence to any Author Accepted Manuscript version arising.
\item No new data were generated or analysed during this study.                                                       \end{itemize}

\phantomsection
\printbibliography
\crefalias{section}{appendix}
\begin{appendices}
\section{Interpolation and discretisation in Bochner spaces}\label{AppendixInter}
In this Section, we collect some useful properties of the disrectisation and interpolation operators introduced in \cref{PrelimDI}.
These properties are quite well-known. Since we could not find the complete statements and proofs in a single source, we include them here for the sake of completeness.

\begin{Lemma}\label{LemComparison}
 Let $n\in\N$, $T\in \frac{1}{n}\N$, $X$ be some Banach space, $u,u_n:\frac{1}{n}\N_0\cap[0,T]\to X$, and $\widetilde u\in \LpX$ for some $p\in[1,\infty]$.
 Then
 \begin{subequations}\label{estPiLa}
 \begin{align}
  &\|\uPi_n u\|_{\LpX}=
  \begin{cases}\left(\frac{1}{n}\sum_{k=0}^{nT-1}\left\|u\left(\frac{k}{n}\right)\right\|_{X}^p\right)^{\frac{1}{p}}&\text{for }p\in[1,\infty),\\
  \max\left\{\left\|u\left(\frac{k}{n}\right)\right\|_X:\ k\in\{0,\dots,nT-1\}\right\}&\text{for }p=\infty,                    
\end{cases}
\label{uPinorm}\\
  &\|\oPi_n u\|_{\LpX}=\begin{cases}\left(\frac{1}{n}\sum_{k=1}^{nT}\left\|u\left(\frac{k}{n}\right)\right\|_{X}^p\right)^{\frac{1}{p}}&\text{for }p\in[1,\infty),\\
  \max\left\{\left\|u\left(\frac{k}{n}\right)\right\|_X:\ k\in\{1,2,\dots,nT\}\right\}&\text{for }p=\infty,                    
\end{cases}\label{oPinorm}\\
&\|\Lambda_n u\|_{\LpX}\leq \begin{cases}
\left(\|\oPi_n u\|_{\LpX}^p+\frac{1}{2n}\|u_n(0)\|_X^p\right)^{\frac{1}{p}}&\text{for }p\in[1,\infty),\\
\max\left\{\left\|u\left(\frac{k}{n}\right)\right\|_X:\ k\in\{0,1,\dots,nT\}\right\}&\text{for }p=\infty,  
\end{cases}
\label{LaPinorm}\\
  &\|\Lambda_n u-\uPi_nu\|_{\LpX}
  =
  \|\oPi_nu-\Lambda_n u\|_{\LpX}
  =C_p\|\oPi_n u- \uPi_nu\|_{\LpX},\nonumber\\ 
  &\text{where }C_p:=\begin{cases}
  \frac{1}{(p+1)^{\frac{1}{p}}}&\text{for }p\in[1,\infty),\\
   1&\text{for }p=\infty;                                           \end{cases}
\label{Lanorm}
 \end{align}
 \end{subequations}
 \begin{align}
   \|\widetilde u\|_{\LpX}\geq 
   \begin{cases}\left(\frac{1}{n}\sum_{k=1}^{nT}\left\|P_n\widetilde u\left(\frac{k}{n}\right)\right\|_X^p\right)^{\frac{1}{p}}&\text{for }p\in[1,\infty),\\
    \max\left\{\left\|P_n\widetilde u\left(\frac{k}{n}\right)\right\|_X:\ k\in\{1,2,\dots,nT\}\right\}&\text{for }p=\infty;
   \end{cases}\label{PiP_}
 \end{align}
 \begin{align}
  &\|\T_nP_n\widetilde u-\widetilde u\|_{\LpX}\limn 0\qquad
  \text{for }\begin{cases}
     p\in[1,\infty),\\
     p=\infty\text{ and }\widetilde u\in C([0,T];X)                                    \end{cases}
\text{ and }\T\in\{\uPi,\oPi,\Lambda\}.\label{convPn}
 \end{align}
 If $p\in[1,\infty)$, then
 \begin{subequations}\label{equivp}
 \begin{align}
  &\quad\|\uPi_n u_n-\widetilde u\|_{\LpX}\limn 0\quad\text{and}\quad\frac{1}{n^{\frac{1}{p}}}\|u_n(T)\|_X\limn0\nonumber\\
  \Leftrightarrow&\quad\|\oPi_n u_n-\widetilde u\|_{\LpX}\limn 0\quad\text{and}\quad\frac{1}{n^{\frac{1}{p}}}\|u_n(0)\|_X\limn0\label{convoPi}\\
  \Rightarrow&\quad\|\Lambda_n u_n-\widetilde u\|_{\LpX}\limn 0.\label{convLuoPi}
 \end{align}
\end{subequations}
If $X$ is a Hilbert space and $p=2$, then
\begin{align}
 \|\Lambda_nu\|_{\LtwoX}\geq \frac{1}{\sqrt{6}}\left(\|\uPi_n u\|_{\LtwoX}^2+\|\oPi_n u\|_{\LtwoX}^2\right)^{\frac{1}{2}}\label{Lanorm2}
\end{align}
and
\begin{subequations}\label{equiv2}
 \begin{align}
  &\quad\|\uPi_n u_n-\widetilde u\|_{\LtwoX}\limn 0\quad\text{and}\quad\frac{1}{n}\|u_n(T)\|_X^2\limn0\nonumber\\
  \Leftrightarrow&\quad\|\oPi_n u_n-\widetilde u\|_{\LtwoX}\limn 0\quad\text{and}\quad\frac{1}{n}\|u_n(0)\|_X^2\limn0\label{convoPi2}\\
  \Leftrightarrow&\quad\|\Lambda_n u_n-\widetilde u\|_{\LtwoX}\limn 0.\label{equivLa2}
 \end{align}
\end{subequations}
Finally, if $u:\frac{1}{n}\N_0\cap[0,T]\to [0,\infty)$, then
\begin{align}
 \|\Lambda_n u\|_{L^{\infty}((0,T))}=\max\left\{\left\|u\left(\frac{k}{n}\right)\right\|_X:\ k\in\{0,1,\dots,nT\}\right\}.\label{LaLInf}
\end{align}
\end{Lemma}
\begin{proof}
\begin{proofpart}[\cref{estPiLa,PiP_,Lanorm2,LaLInf}]
Identities \cref{uPinorm,oPinorm,LaLInf} are evident. 

Let $p\in[1,\infty)$. We compute using the convexity of $s\mapsto s^p$ and \cref{oPinorm} that 
\begin{align}
 \int_0^T\|\Lambda_n u\|_X^p\,dt=&\sum_{k=1}^{nT}\int_{\frac{k-1}{n}}^{\frac{k}{n}}\left\|\frac{\frac{k}{n}-t}{\frac{1}{n}}u\left(\frac{k-1}{n}\right)+\frac{t-\frac{k-1}{n}}{\frac{1}{n}}u\left(\frac{k}{n}\right)\right\|_X^p\,dt\nonumber\\
 \leq&\sum_{k=1}^{nT}\int_{\frac{k-1}{n}}^{\frac{k}{n}}\left(\frac{\frac{k}{n}-t}{\frac{1}{n}}\left\|u\left(\frac{k-1}{n}\right)\right\|_X+\frac{t-\frac{k-1}{n}}{\frac{1}{n}}\left\|u\left(\frac{k}{n}\right)\right\|_X\right)^p\,dt\nonumber\\
 \leq&\sum_{k=1}^{nT}\int_{\frac{k-1}{n}}^{\frac{k}{n}}\frac{\frac{k}{n}-t}{\frac{1}{n}}\left\|u\left(\frac{k-1}{n}\right)\right\|_X^p+\frac{t-\frac{k-1}{n}}{\frac{1}{n}}\left\|u\left(\frac{k}{n}\right)\right\|_X^p\,dt\nonumber\\
 =&\frac{1}{2n}\sum_{k=1}^{nT}\left(\left\|u\left(\frac{k-1}{n}\right)\right\|_X^p+\left\|u\left(\frac{k}{n}\right)\right\|_X^p\right)\nonumber\\
 =&\frac{1}{n}\sum_{k=1}^{nT}\left\|u\left(\frac{k}{n}\right)\right\|_X^p+\frac{1}{2n}\|u_n(0)\|_X^p-\frac{1}{2n}\|u_n(T)\|_X^p\nonumber\\
 \leq&\|\oPi_n u\|_{\LpX}^p+\frac{1}{2n}\|u_n(0)\|_X^p,\nonumber 
\end{align}
yielding \cref{LaPinorm}. Identity \cref{Lanorm} follows with 
 \begin{align}
 \int_0^T\|\uPi_n u-\Lambda_n u\|_X^p\,dt  =&\sum_{k=1}^{nT}\int_{\frac{k-1}{n}}^{\frac{k}{n}}\left\|-\frac{u\left(\frac{k}{n}\right)-u\left(\frac{k-1}{n}\right)}{\frac{1}{n}}\left(t-\frac{k}{n}\right)\right\|_X^p\,dt\nonumber\\
  =&\frac{1}{p+1}\frac{1}{n}\sum_{k=1}^{nT}\left\|u\left(\frac{k}{n}\right)-u\left(\frac{k-1}{n}\right)\right\|_X^p\nonumber\\
  =&\frac{1}{p+1}\int_0^T\|\oPi_nu-\uPi_n u\|_X^p\,dt\nonumber
 \end{align}
 and a similar computation for $\Lambda_nu-\uPi_nu$. Passing to the limit in \cref{LaPinorm,Lanorm} as $p\to\infty$, we conclude that both of them hold for $p=\infty$ as well.
 
 Let $p\in[1,\infty)$. Estimating the norm of the arising Bochner's integral and then using H\"older's inequality, we obtain
 \begin{align}
  \frac{1}{n}\sum_{k=1}^{nT}\left\|P_n\widetilde u\left(\frac{k}{n}\right)\right\|_X^p=&n^{p-1}\sum_{k=1}^{nT}\left\|\int_{\frac{k-1}{n}}^{\frac{k}{n}}\widetilde u(t)\,dt\right\|_X^p\nonumber\\
  \leq&\sum_{k=1}^{nT}\int_{\frac{k-1}{n}}^{\frac{k}{n}}\|\widetilde u(t)\|_X^p\,dt\nonumber\\
  =&\int_0^T\|\widetilde u(t)\|_X^p\,dt,\nonumber
 \end{align}
 yielding \cref{PiP_} for finite $p$. Passing to the limit in  \cref{PiP_} as $p\to\infty$, we confirm   that the inequality holds for $p=\infty$ as well.
 
 For a  Hilbert space $(X,(\cdot,\cdot)_X)$, we compute using the Cauchy-Schwartz inequality and  \cref{uPinorm,oPinorm} 
 \begin{align}
  \sum_{k=1}^{nT}\int_{\frac{k-1}{n}}^{\frac{k}{n}}\|\Lambda_n u\|_X^2\,dt
  =&\sum_{k=1}^{nT}\int_{\frac{k-1}{n}}^{\frac{k}{n}}\left\|\frac{\frac{k}{n}-t}{\frac{1}{n}}u\left(\frac{k-1}{n}\right)+\frac{t-\frac{k-1}{n}}{\frac{1}{n}}u\left(\frac{k}{n}\right)\right\|_X^2\,dt\nonumber\\
  =&n^2\sum_{k=1}^{nT}\int_{\frac{k-1}{n}}^{\frac{k}{n}}\left(\frac{k}{n}-t\right)^2\left\|u\left(\frac{k-1}{n}\right)\right\|_X^2+\left(t-\frac{k-1}{n}\right)^2\left\|u\left(\frac{k}{n}\right)\right\|_X^2\nonumber\\
  &\phantom{n^2\sum_{k=1}^{nT}\int_{\frac{k-1}{n}}^{\frac{k}{n}}}+2\left(\frac{k}{n}-t\right)\left(t-\frac{k-1}{n}\right)\left(u\left(\frac{k-1}{n}\right),u\left(\frac{k}{n}\right)\right)_X\,dt\nonumber\\
  =&\frac{1}{3n}\sum_{k=1}^{nT}\left(\left\|u\left(\frac{k-1}{n}\right)\right\|_X^2+\left\|u\left(\frac{k}{n}\right)\right\|_X^2+\left(u\left(\frac{k-1}{n}\right),u\left(\frac{k}{n}\right)\right)_X\right)\nonumber\\
  \geq&\frac{1}{6n}\sum_{k=1}^{nT}\left(\left\|u\left(\frac{k-1}{n}\right)\right\|_X^2+\left\|u\left(\frac{k}{n}\right)\right\|_X^2\right)\nonumber\\
  =&\frac{1}{6}\left(\|\uPi_n u\|_{\LtwoX}^2+\|\oPi_n u\|_{\LtwoX}^2\right),\nonumber
 \end{align}
so \cref{Lanorm2} follows. 
\end{proofpart}
\begin{proofpart}[\cref{convPn}]
Let us first assume $$p=\infty\text{ and }\widetilde u\in C([0,T];X).$$ In this case, estimating the norm of the arising Bochner's integral, we find that  
\begin{align}
&\max\left\{\underset{t\in\left(\frac{k-1}{n},\frac{k}{n}\right)}{\sup}\left\|P_n\widetilde u\left(\frac{k}{n}\right)-\widetilde u(t)\right\|_X:\ k\in\{1,2,\dots, nT\}\right\}\nonumber\\
=&\max\left\{n\underset{t\in\left(\frac{k-1}{n},\frac{k}{n}\right)}{\sup}\left\|\int_{\frac{k-1}{n}}^{\frac{k}{n}}(\widetilde u(s)-\widetilde u(t))\,ds\right\|_X:\ k\in\{1,2,\dots, nT\}\right\}\nonumber\\
 \leq&\max\left\{n\underset{t\in\left(\frac{k-1}{n},\frac{k}{n}\right)}{\sup}\int_{\frac{k-1}{n}}^{\frac{k}{n}}\|\widetilde u(s)-\widetilde u(t)\|_X\,ds:\ k\in\{1,2,\dots, nT\}\right\}\nonumber\\
 \leq&\max\left\{ \|\widetilde u(s)-\widetilde u(t)\|_X:\ s,t\in[0,T]\text{ and }|t-s|\leq \frac{1}{n}\right\}\nonumber\\
 & \limn0,\nonumber
\end{align}
proving \cref{convPn} for $\T=\oPi$.
A similar calculation yields \cref{convPn} for $\T=\uPi$. 

Now let $$p\in[1,\infty)\text{ and }\widetilde u\in C([0,T];X).$$ 
Since $[0,T]$ is finite,  \cref{convPn} for $\T\in\{\uPi,\oPi\}$ and $p\in[1,\infty)$ directly follows from the corresponding convergences for $p=\infty$.  

We turn to the general case of $$p\in[1,\infty)\text{ and }\widetilde u\in \LpX.$$ Consider a sequence $(\widetilde u_m)\subset C([0,T];X)$ such that 
\begin{align*}
 \widetilde u_m\underset{m\to\infty}{\to} \widetilde u\qquad\text{in }\LpX.
\end{align*}
 Combining \cref{oPinorm} and \cref{PiP_} with \cref{convPn} for (continuous) $\widetilde u_m$ and $\T=\oPi$, we obtain
\begin{align}
 &\underset{n\to\infty}{\lim\sup}\|\oPi_nP_n\widetilde u-\widetilde u\|_{\LpX}\nonumber\\
 \leq& \underset{m\to\infty}{\lim\sup}\,\underset{n\to\infty}{\lim\sup}\left(\|\oPi_nP_n\widetilde u_m-\widetilde u_m\|_{\LpX}+\|\widetilde u_m-\widetilde u\|_{\LpX}+\|\oPi_n P_n(\widetilde u_m-\widetilde u)\|_{\LpX}\right)\nonumber\\
 \leq& \underset{m\to\infty}{\lim\sup}\,\underset{n\to\infty}{\lim\sup}\left(\|\oPi_nP_n\widetilde u_m-\widetilde u_m\|_{\LpX}+2\|\widetilde u_m-\widetilde u\|_{\LpX}\right)\nonumber\\
 =&0.\nonumber
\end{align}
A proof for $\T=\uPi$ is analogous. Since $\Lambda_n$ is a convex combination of $\uPi_n$ and $\oPi_n$ on each interval $((k-1)/n,k/n)$ by \cref{convComb}, 
\cref{convPn} holds also for $\T=\Lambda$.

\end{proofpart}
\begin{proofpart}[\cref{equivp,equiv2}]
Thanks to \cref{convPn}, we have
\begin{align}
 \|\T_n u_n-\widetilde u\|_{\LpX}\limn 0\quad  \Leftrightarrow\quad\|\T_n (u_n-P_n\widetilde u)\|_{\LpX}\limn 0\qquad\text{for }\T\in\{\uPi,\oPi,\Lambda\}.\label{equiv11}
\end{align}
Further, estimating the norm of the arising Bochner's integral and using H\"older's inequality, we find that 
\begin{align}
 \frac{1}{n^{\frac{1}{p}}}\|P_n\widetilde u(t)\|_X=&n^{1-\frac{1}{p}}\left\|\int_{t-\frac{1}{n}}^t\widetilde u(t)\,dt\right\|_X\nonumber\\
 \leq&n^{1-\frac{1}{p}}\int_{t-\frac{1}{n}}^t\|\widetilde u(t)\|_X\,dt
 \nonumber\\
 \leq&\int_{t-\frac{1}{n}}^t\|\widetilde u(t)\|_X^p\,dt\nonumber\\
 \limn &0\qquad\text{for }t\in\{1,T\}.\label{equivT}
\end{align}
Due to \cref{equivT,equiv11}, it suffices to prove \cref{equivp,equiv2} for $\widetilde u=0$, i.e. that 
\begin{subequations}
 \begin{align}
  &\quad\|\uPi_n u_n\|_{\LpX}\limn 0\quad\text{and}\quad\frac{1}{n^{\frac{1}{p}}}\|u_n(T)\|_X\limn0\nonumber\\
  \Leftrightarrow&\quad\|\oPi_n u_n\|_{\LpX}\limn 0\quad\text{and}\quad\frac{1}{n^{\frac{1}{p}}}\|u_n(0)\|_X\limn0\label{0convoPi}\\
  \Rightarrow&\quad\|\Lambda_n u_n\|_{\LpX}\limn 0,\label{0convLuoPi}
 \end{align}
\end{subequations}
and, for $p=2$ and $X$ a Hilbert space, also the reverse implication 
\begin{align}
  &\quad\|\Lambda_n u_n\|_{\LpX}\limn 0\nonumber\\
  \Rightarrow&\quad\|\oPi_n u_n\|_{\LpX}\limn 0\quad\text{and}\quad\frac{1}{n^{\frac{1}{p}}}\|u_n(0)\|_X\limn0\label{0convoPirev}
\end{align}
 The equivalence in \cref{0convoPi} is a direct consequence of \cref{uPinorm,oPinorm}. Combining \cref{0convoPi,Lanorm}, we obtain implication \cref{0convLuoPi}.   
Finally, implication  \cref{0convoPirev} directly follows with \cref{uPinorm,oPinorm,Lanorm2}.
 \end{proofpart}
 
\end{proof}
\begin{Remark}
\cref{LemComparison} is motivated by \cite[Chapter 11 Lemma 11.4]{SchweizerPDEs}. There, the implication 
\begin{align}
 \|\Lambda_n u_n-\widetilde u\|_{\LtwoX}\to 0\qquad\Rightarrow \qquad\|\oPi_n u_n-\widetilde u\|_{\LtwoX}\to 0\nonumber
\end{align}
was proved with the help of an estimate similar to \cref{Lanorm2}.
\end{Remark}

\begin{Lemma}\label{LemIneq}
 Let $n\in\N$, $T\in \frac{1}{n}\N$, $u:\frac{1}{n}\N_0\cap[0,T]\to \R$, and $g:\R\rightarrow\R$ be an increasing function with a primitive $G$.  Then
 \begin{align}
  \partial_t\Lambda_nug\left(\oPi_n{u}\right)\geq&\partial_t \Lambda_n{G(u)}\qquad\text{in } [0,T]\backslash\left(\frac{1}{n}\N_0\right).\label{propg}
 \end{align}
\end{Lemma}
The proof is a consequence of the following elementary inequality: under the assumptions on $g$ from the Lemma it holds that 
 \begin{align}
  (b-a)g(b)\geq G(b)-G(a)\qquad\text{for all }a,b\in\R. \label{ineqG}
 \end{align} 
\begin{proof}[{\it {\bf Proof of} \cref{LemIneq}.}]
 Thanks to \cref{ineqG}
 it holds that
 \begin{align}
  \partial_t\Lambda_nug\left(\oPi_n{u}\right)=&\frac{u\left(\frac{k}{n}\right)-u\left(\frac{k-1}{n}\right)}{\frac{1}{n}}g\left(u\left(\frac{k}{n}\right)\right)\nonumber\\
  \geq&\frac{G\left(u\left(\frac{k}{n}\right)\right)-G\left(u\left(\frac{k-1}{n}\right)\right)}{\frac{1}{n}}\nonumber\\
  =&\partial_t \Lambda_n{G(u)}\qquad\text{in } \left(\frac{k-1}{n},\frac{k}{n}\right)\text{ for }k\in \{1,2,\dots,nT\}.\nonumber
 \end{align}
\end{proof}
We conclude this Section with a version of the discrete Gronwall lemma.
\begin{Lemma}\label{LemDiscGr}
 Let $n\in\N$, $T\in \frac{1}{n}\N$, and $C\in(-\infty,n)$. Define
 \begin{align}
 &\exp_{n,C}:\frac{1}{n}\N_0\to (0,\infty),\qquad \exp_{n,C}(t):=\left(\left(1-\frac{C}{n}\right)^{-n}\right)^t.\nonumber
\end{align}
 Let $u,v:\frac{1}{n}\N_0\cap[0,T]\to [0,\infty)$.  Then
 \begin{align}
  \partial_t\Lambda_n u\leq C \oPi_n u+\oPi_n v\qquad\text{in } [0,T]\backslash\left(\frac{1}{n}\N_0\right)\label{ineqGro}
 \end{align}
implies 
\begin{align}
 \|\oPi_n \exp_{n,C}^{-1}u\|_{L^{\infty}(0,T)}\leq u(0)+\|\uPi_n\exp_{n,C}^{-1}\oPi_nv\|_{L^1((0,T))}.\label{DiscGrin}
\end{align}

\end{Lemma}
\begin{proof}
Set 
\begin{align*}
a_n:=u\left(\frac{k}{n}\right),\quad g_n:=v\left(\frac{k}{n}\right),\quad \Delta t:=\frac{1}{n}.                                                                                                                                      \end{align*}
In this notation, inequality \cref{ineqGro} can be rewritten as
\begin{align}
 \frac{a_{n+1}-a_n}{\frac{1}{n}}\leq C a_{n+1}+g_{n+1}\qquad\text{for }n\in\{0,\cdots,nT-1\}.\nonumber
\end{align}
Since $1-C/n>0$, a discrete version of Gronwall's lemma   \cite[Proposition 3.1]{Emmrich} can be applied in this setting, followed by division throughout by $\exp_{n,C}(k/n)$. The result is 
\begin{align}
 \exp_{n,C}^{-1}u\left(\frac{k}{n}\right)\leq&u(0)+\frac{1}{n}\sum_{j=0}^{k-1}\exp_{n,C}^{-1}\left(\frac{j}{n}\right)v\left(\frac{j+1}{n}\right)\qquad\text{for }k\in\{1,\dots,nT\}.\label{DiscGrin_1}
\end{align}
Since $v$ is non-negative,  \cref{DiscGrin_1} implies 
\begin{align}
 \exp_{n,C}^{-1}u\left(\frac{k}{n}\right)\leq&u(0)+\frac{1}{n}\sum_{j=0}^{nT-1}\exp_{n,C}^{-1}\left(\frac{j}{n}\right)v\left(\frac{j+1}{n}\right)\qquad\text{for }k\in\{1,\dots,nT\}.\label{DiscGrin_}
\end{align}
Re-writing \cref{DiscGrin_} in terms of continuous-time interpolations gives \cref{DiscGrin}.

\end{proof}

\section{Auxiliary results about functions and measures with values in Banach spaces}\label{AppVM}
The two results in this Section help dealing with functions and vector measures that take values in Banach spaces. Our first Lemma is about representation of such measures by weakly-$*$ measurable functions. It relies on some known facts related to vector measures and  presents them in the form that is useful for our analysis in \cref{Secsolc}. 
\begin{Lemma}\label{LemRep}
Let $I$ be a finite  interval and $X$ be a Banach space.
\begin{enumerate}
[label=(\arabic*),ref=(\arabic*)]
\item 
Each $f\in L^1_{w-*}(I;X^*)$  generates a vector measure  $\mu_f\in {\cal M}(I;X^*)$ that is given by the Gelfand integral \cite[Chapter II \S3 Definition 2]{DiestelUhl} of $f$, i.e.
 \begin{align}
  \left<\mu_f(E),\varphi\right>:=\int_E\left<f,\varphi\right>dt\qquad\text{for all }\varphi\in X\text{ and Borel }E\subset I\label{Defmu}
 \end{align}
 and satisfies
 \begin{align}
  \|\mu_f(E)\|_{X^*}\leq \int_E\|f\|_{X^*}\,dt\qquad\text{for all Borel }E\subset I.\label{normmuE}
 \end{align}        
 \item\label{LemRep2}  Let $\mu \in {\cal M}(I;X^*)$ and assume that there exists $0\leq g\in L^1(I)$ such that
 \begin{align}
  \|\mu(E)\|_{X^*}\leq \int_Eg\,dt\qquad\text{for all Borel }E\subset I.\label{majg}
 \end{align} 
 Then there exists some $ \zeta\in L^{\infty}_{w-*}(I;X^*)$ such that 
 \begin{align}
&\|\zeta\|_{L^{\infty}_{w-*}(I;X^*)}\leq1\nonumber               \end{align}
and
\begin{align}
   \mu=\mu_{\zeta g}.\nonumber
\end{align}
If $X$ is separable, then any $\zeta_1$ and $\zeta_2$ that satisfy the conditions on $\zeta$ coincide a.e. in $\{g>0\}$.

 \end{enumerate}

\end{Lemma}
\begin{proof}
 \begin{enumerate}
 [label=(\arabic*),ref=(\arabic*)]
  \item Since $f\in L^1_{w-*}(I;X^*)$, its Gelfand integral exists. Hence, $\mu_f$ is a well-defined $X^*$-valued set function on the set of Borel subsets of $I$. Estimate \cref{normmuE} is a trivial consequence of \cref{Defmu} and the definition of the $X^*$-norm. Countable additivity directly follows with the corresponding property of the Lebesgue integral of scalar functions, whereas regularity and variation finiteness are direct consequences of the corresponding properties of the scalar measure defined by the right-hand side of the inequality in \cref{normmuE}. 
  \item Thanks to \cref{majg} the given vector measure $\mu$ is continuous with respect to the finite measure $g\,dt$.  By
\cite[Chapter II \S13 4.  Theorem 5]{Dinculeanu} (take there $T:=I$, $E:=X$, $F:=Z:=\R$, $\nu:=g\,dt$, and $m:=\mu$) there exists some $\zeta\in L^1_{w-*,loc}(I;X^*)$ such that $\mu=\mu_{\zeta g}$ and (take $\varphi$ in that Theorem to be the characteristic function of $E$)
 \begin{align}
  \int_E\|\zeta\|_{X^*}\, g\,dt=|\mu|(E)\qquad\text{for all Borel }E\subset I,\label{normeq}
 \end{align}
 where by $|\mu|(E)$ we denote  the variation of $\mu$ over $E$.
Combining \cref{normeq,majg}, we obtain
\begin{align}
  \int_E\|\zeta\|_{X^*}\, g\,dt\leq \int_Eg\,dt\qquad\text{for all Borel }E\subset I,\nonumber
 \end{align}
 and this implies $\|\zeta\|_{X^*}\leq 1$ a.e. in $I$, as required.

 Now let $X$ be separable, ${\cal C}$ being a countable dense subset of $X$. Assume that $\zeta_1$ and $\zeta_2$ satisfy the conditions on $\zeta$. Then \cref{Defmu} for $f_i:=\zeta_i g$ leads to 
 \begin{align}
  \int_E\left<\zeta_1,\varphi\right>gdt=\int_E\left<\zeta_2,\varphi\right>gdt\qquad\text{for all }\varphi\in {\cal C}\text{ and Borel }E\subset I.\nonumber
 \end{align}
Consequently, 
\begin{align}
\left<(\zeta_1-\zeta_2),\varphi\right>g=0\qquad \text{ a.e. in } I \text{ for all }\varphi\in {\cal C}. \nonumber                                                                               \end{align}
 Since ${\cal C}$ is countable, we conclude  that 
 \begin{align}
\left<(\zeta_1-\zeta_2),\varphi\right>g=0\qquad \text{ for all }\varphi\in {\cal C}\text{ a.e. in } I.   \nonumber                                                                          \end{align}
 Finally, utilising the density of  ${\cal C}$ in $X$, we deduce that $\zeta_1-\zeta_2=0$ a.e. in $\{g>0\}$.
 \end{enumerate}
 
\end{proof}

The next Lemma deals specifically with weak measurability of parts of a time-dependent Radon measure.
\begin{Lemma}\label{LemMofParts}
 Let $I$ be a finite interval and $O$ be a bounded domain. Let $f:I\to \RM$ be weakly  measurable. Then: $f^{ac}$ is   $\LOne$-measurable, $f^s$ is weakly $\RM$-measurable, and $|f^s|$ is  weakly-$*$ $\RM$-measurable. 
\end{Lemma}
\begin{proof}
 Since operators $(\cdot)^{ac}:\RM\to \LOne$ and $(\cdot)^s:\RM\to \RM$ are linear and bounded, we conclude that $f^{ac}:I\to \LOne$ and $f^s:I\to \RM$ are weakly  measurable. Moreover,  since $\LOne$ is separable,  Pettis' theorem applies and yields the strong measurability of $f^{ac}$.  
 
 In order to prove that $|f^s|$ is weak-$*$ measurable it suffices to verify that $\left<|f^s|,\varphi\right>$ is measurable for any  $0\leq\varphi\in \Co$. Observe that for such $\varphi$
 \begin{align}
  \left<|f^s(t,\cdot)|,\varphi\right>=&\left\|\varphi f^s(t,\cdot)\right\|_{\RM}\nonumber\\
  =&{\sup}\left\{\left<f^s(t,\cdot),\varphi\psi\right>:\ \psi\in\Co\text{ such that } \|\psi\|= 1\right\}.\label{sup1}
\end{align}
Since $\Co$ is separable, the supremum in \cref{sup1} can be replaced by the supremum over a countable set. Exploiting the weak measurability of $f^s$, we conclude that the function on the right-hand side of \cref{sup1} is measurable. The proof of the Lemma is thus complete. 
\end{proof}

\end{appendices}
\end{document}